\documentclass[1 [leqno, 11pt]{amsart}
\usepackage{amssymb,  amsmath, amsmath, latexsym, amssymb, amsfonts, amsbsy, amsthm,mathtools, graphicx, CJKutf8, CJKnumb, CJKulem, extarrows, color, dsfont}
\usepackage{mathtools,graphicx,CJKutf8,CJKnumb,CJKulem,color}
\usepackage{cancel}
\usepackage{hyperref}
\usepackage{scalerel,stackengine}
\usepackage{float}
\usepackage{multirow}
\usepackage{makecell}
\usepackage{tikz}
\usetikzlibrary{shapes.geometric, arrows,plotmarks,backgrounds}
\usepackage{pgfplots}
\numberwithin{equation}{section}

\allowdisplaybreaks

\let\al=\alpha
\let\b=\beta
\let\g=\gamma

\let\e=\varepsilon

\let\s=\sigma
\let\f=\frac

\let\om=\omega

\let\D=\Delta

\let\ep=\epsilon
\let\Om=\Omega
\let\wt=\widetilde
\let\wh=\widehat
\let\ol=\overline

\let\na=\nabla
\let\th=\theta
\let\pa=\partial
\let\va=\varphi
\let\Th=\Theta
\let\vth=\vartheta

\def\cA{\mathcal{A}}

\def\cF{\mathcal{F}}

\def\cK{\mathcal{K}}

\def\cM{\mathcal{M}}
\def\cN{\mathcal{N}}

\def\eqdef{\buildrel\hbox{\footnotesize def}\over =}

\def\bbT{\mathbb{T}}
\def\bbZ{\mathbb{Z}}
\def\bbR{\mathbb{R}}

\newcommand{\beq}{\begin{equation}}
\newcommand{\eeq}{\end{equation}}
\newcommand{\beqno}{\begin{equation*}}
\newcommand{\eeqno}{\end{equation*}}
\newcommand{\ben}{\begin{eqnarray}}
\newcommand{\een}{\end{eqnarray}}
\newcommand{\beno}{\begin{eqnarray*}}
\newcommand{\eeno}{\end{eqnarray*}}

\newtheorem{theorem}{Theorem}[section]

\newtheorem{lemma}[theorem]{Lemma}
\newtheorem{proposition}[theorem]{Proposition}

\newtheorem{remark}[theorem]{Remark}

\begin{document}
\begin{CJK*}{GBK}{song}

\title[Stability threshold of the Couette flow]{Stability threshold of the Couette flow for Navier-Stokes Boussinesq system with large Richardson number $\g^2>\frac{1}{4}$}
\author{Cuili Zhai}
\address{School of Mathematics and Physics, University of Science and Technology Beijing, 100083, Beijing, P. R. China.}
\email{zhaicuili035@126.com, cuilizhai@ustb.edu.cn}
\author{Weiren Zhao}
\address{Department of Mathematics, New York University Abu Dhabi, Saadiyat Island, P.O. Box 129188, Abu Dhabi, United Arab Emirates.}
\email{zjzjzwr@126.com, wz19@nyu.edu}

\begin{abstract}
In this paper, we study the nonlinear asymptotic stability of the Couette flow in the stably stratified regime, namely the Richardson number $\g^2>\frac{1}{4}$. Precisely, we prove that if the initial  perturbation $(u_{in},\vth_{in})$ of the Couette flow $v_s=(y,0)$ and the linear temperature $\rho_s=-\g^2y+1$ satisfies $\|u_{in}\|_{H^{s+1}}+\|\vth_{in}\|_{H^{s+2}}\leq \ep_0\nu^{\f12}$, then the asymptotic stability holds.
\end{abstract}
\maketitle

\section{Introduction}
The stability of shear flow in a stratified medium is of interest in many fields of research such as fluid dynamics, geophysics, astrophysics, mathematics, etc. Density stratification can strongly affect the dynamic of fluids like air in the atmosphere or water in the ocean and the stability of the question of stratified flows dates back to Taylor \cite{Taylor} and Goldstein \cite{Goldstein}, and since then there has been an active search towards the understanding of the stability of density-stratified flows. 
\subsection{Two dimensional Navier- Stokes Boussinesq equations}
In this paper, we consider the two-dimensional Navier- Stokes Boussinesq system with full dissipation  in $\Om=\bbT\times \bbR$:
 \beq\label{eq: NSB}
 \left\{\begin{array}{l}
\pa_tv+v\cdot\na v-\nu\D v+\na P=-\rho ge_2\\
\pa_t\rho+v\cdot\na\rho-\mu\D \rho=0, \\
\na\cdot v=0,\\
v|_{t=0}=v_{in}(x,y), \ \rho|_{t=0}=\rho_{in}(x,y),
\end{array}\right.
\eeq
where $(x,y)\in\Om$, $v=(v^1,v^2)$ is the velocity field, $P$ is the pressure, $\rho$ is the temperature (density), $g$ is the gravitational constant, $e_2=(0,1)$ is the unit vector in the vertical direction, $\nu$ is the viscosity coefficient and $\mu$ is the thermal diffusivity. In this paper, we focus on the case when $C^{-1}\leq \f{\mu}{\nu}\leq C$ with some constant $C>1$ independent of $\nu$. Thus, the three parameters $\mu,\nu$ and $g$ can be normalized to $\mu=\nu$ and $g=1$.

The system admits a class of  steady states, where the velocity field is the Couette flow and the temperature is a linear function of the vertical component, namely,
\beq\label{steady state}
v_s=(y,0),\quad \rho_s=-\g^2 y+1,\quad p_s=\f12\g^2y^2-y+C_0,
\eeq
where the constant $\g^2$ is the Richardson number \footnote{In general, the Richardson number $\g^2(y)=-\f{g\pa_y\rho_s(y)}{(\pa_yU(y))^2}$ is not always a constant, where $g$ is the gravitational constant, $U(y)$ is the horizontal velocity of the shear flow, and $\rho_s$ is the temperature.}. The Richardson number is one of the control parameters of the stability of stratified shear flows. The Miles-Howard theorem \cite{Howard, Miles} guarantees that any flow in the inviscid non-diffusive limit is linearly stable if the local Richardson number everywhere exceeds the value $\f14$, however, unstable modes can arise when the Richardson number is smaller than $\f14$ \cite{Drazin}.

 In this paper, we study the asymptotic stability of the Couette flow when the Richardson number is greater than $\f14$, namely,
 $$
 \g^2>\f14.
 $$
 
 It is natural to introduce the perturbation. Let $v=u+v_s$, $P=p+p_s$ and $\rho=\g^2\vth+\rho_s$, then $(u,p,\vth)$ satisfies
\begin{equation}\label{eq: u sys}
 \left\{\begin{array}{l}
\pa_tu+y\pa_xu+\Big(
  \begin{array}{ccc}
   u^2\\
    0\\
  \end{array}
\Big)+u\cdot\na u-\nu\D u+\na p=- \Big(\begin{array}{ccc}
   0\\
    \g^2\vth\\
  \end{array}\Big),\\
\pa_t\vth+y\pa_x\vth+u\cdot\na\vth-\nu\D \vth-u^2=0,\ 
\na\cdot u=0,\\
u|_{t=0}=u_{in}(x,y),\ \vth|_{t=0}=\vth_{in}(x,y).
\end{array}\right.
\end{equation}

Let $\om=\na \times u=\pa_xu^2-\pa_yu^1$ be the vorticity, then $(\om,\vth)$ satisfies
\beq\label{eq: vorticity}
 \left\{\begin{array}{l}
\pa_t\om+y\pa_x\om+u\cdot\na \om-\nu\D \om=-\g^2\pa_x\vth,\\
\pa_t\vth+y\pa_x\vth+u\cdot\na\vth-\nu\D\vth-u^2=0,\\
u=\na^{\bot}\psi=(-\pa_y\psi,\pa_x\psi),\quad \D\psi=\om.
\end{array}\right.
\eeq

\subsection{Historical comments}
In the physics literature, there has been a lot of work devoted to the stability of the Couette flow in the linearized stratified inviscid flow. See \cite{Chimonas,Farrell} and the references therein. There are only few mathematically rigorous results.

Lin and Yang \cite{YangLin} studied the linear asymptotic stability of the steady-state \eqref{steady state} for the 2D Euler Boussinesq system. More precisely, they proved that if $\g^2>\f14$, then the solution to the system \eqref{Euler-Boussinesq} with $\nu=0$ satisfies
\begin{align}
(u^1_{\neq}, u^2_{\neq},\om_{\neq},\vth_{\neq})&\lesssim\Big(\langle t\rangle^{-\f12}, \langle t\rangle^{-\f32},\langle t\rangle^{\f12},\langle t\rangle^{-\f12}\Big),
\end{align}
where $P_{\neq}f=f_{\neq}=f-P_0f$ and $P_0f=f_0=\f{1}{2\pi}\int_{\bbT}f(t,x,y)dx$ denote the non-zero mode and the zero mode. See also \cite{Bianchini} for other linear results of general shear flows and \cite{BBZ} for nonlinear results. For the Navier-Stokes Boussinesq system without the heat diffusion ($\nu=1,\ \mu=0$), Masmoudi, Said-Houari and Zhao \cite{MSZ} proved that if the initial perturbations are in Gevrey-$m$ with $1\leq m<3$, then the steady state \eqref{steady state} is asymptotically stable.
For the case $\nu=\mu>0$, in \cite{DWZ,MZZ}, the authors proved the asymptotic stability of the steady-state \eqref{steady state} with $\g^2=0$, if the initial perturbations $(u_{in},\vth_{in})$ satisfy 
\begin{align*}
\|u_{in}\|_{H^2}\leq \e_0\nu^{\f12},\ 
 \|\vth_{in}\|_{H^1}+ \||D_x|^{\f16}\vth_{in}\|_{H^1}\leq \e_1\nu^{\f{11}{12}}, 
\end{align*}
for both the finite channel ($\Om=\bbT\times[-1,1] $) and the infinite channel ($\Om=\bbT\times \bbR$) cases. We also refer to \cite{Zillinger}  for the asymptotic stability result of the Couette flow with a linear temperature.

The mechanism leading to stability is the so-called inviscid damping and enhanced dissipation. Similar phenomena also happen in other fluids systems. One may refer to \cite{BM, IJ, IJ1, MZ} for the inviscid damping results of Euler equations, and to \cite{BVW,CLWZ, LMZ, MZ1,MZ2} for the enhanced dissipation results of Navier-Stokes equations around 2D Couette flow,  and to \cite{BGM1,BGM2,BGM3,WZ,CWZ} for the enhanced dissipation results of Navier-Stokes equations around 3D Couette flow. For other shear flows, one may refer to \cite{LX, IMM, WZZ, LWZ} for Kolmogorov flow and to \cite{ZEW, Z, DL} for Poiseuille flow. 

In this paper, we focus on the small viscosity case and study the stability threshold problem, namely,

{\it{Given norms $\|\cdot\|_{Y_1}$ and $\|\cdot\|_{Y_2}$, find $\al=\al(Y_1,Y_2)$ and $\beta=\beta(Y_1,Y_2)$ such that
\begin{align*}
&\|u_{in}\|_{Y_1}\leq \nu^{\al}\ \text{and} \ \|\vth_{in}\|_{Y_2}\leq \nu^{\beta}\ \Rightarrow \text{stability};\\
&\|u_{in}\|_{Y_1}\gg \nu^{\al}\ \text{or} \ \|\vth_{in}\|_{Y_2}\gg \nu^{\beta}\ \Rightarrow \text{instability}.
\end{align*}}}

\subsection{Main result}
Our main result states as follows:
\begin{theorem}\label{thm}
Let $s\geq 6$ and $\g^2>\f14$. There exist $0<\ep_0=\ep_0(s,\g^2)<1, 0<\nu_0=\nu_0(s,\g^2)<1$, such that for all $0<\nu\leq\nu_0$ and $0<\ep\leq\ep_0$, if  the initial data $(u_{in}, \vth_{in})$ satisfies 
\begin{align*}
\|u_{in}\|_{H^{s+1}}+\|\vth_{in}\|_{H^{s+2}}\leq \ep\nu^{\f12},
\end{align*}
then the solution $(u,\om,\vth)$ of \eqref{eq: vorticity} satisfies
\begin{align}\label{thm: u}
\|P_{\neq}u^1\|_{L^2}+\langle t\rangle\|P_{\neq}u^2\|_{L^2}+\langle t\rangle^{-1}\|P_{\neq}\om\|_{L^2}+\|P_{\neq}\vth\|_{L^2}
&\leq C\ep\nu^{\f12}\langle t\rangle^{-\f12}e^{-c\nu^{\f13}t},\\
\label{zero}
\nu^{-\f14}\|P_0u^1\|_{H^s}+\|P_0\om\|_{H^s}+\nu^{-\f14}\|P_0\vth\|_{H^s}&\leq C\ep\nu^{\f14},\\
\label{w}
\|P_{\neq}\big(|\pa_x|^{\f12}(-\D)^{-\f14}\om\big)(t,x+ty,y)\|_{H^s}&\leq C\ep\nu^{\f12}e^{-c\nu^{\f13}t},\\
\label{th}
\|P_{\neq}\big(|\pa_x|^{\f12}(-\D)^{\f14}\vth\big)(t,x+ty,y)\|_{H^s}&\leq C\ep\nu^{\f12}e^{-c\nu^{\f13}t},
\end{align}
where $C$ and $c$ are the constants independent of $t,\ep$ and $\nu$.
\end{theorem}
\begin{remark}
We consider the case $\g^2-\f14\gg \nu$ in this paper. It is interesting to study the case when $\g^2-\f14$ is $\nu$-related.
\end{remark}
\begin{remark}
For the Navier-Stokes equation, $\rho=0$, in \cite{MZ2}, the authors proved that the asymptotic stability of the Couette flow holds if the initial perturbation satisfies
$ \|\om_{in}\|_{H^{\s}}\leq \ep_0\nu^{\f13}$. Moreover, the enhanced dissipation rate is 
\begin{align*}
\|P_{\neq}\om\|_{L^2}\leq Ce^{-c\nu^{\f13}t}.
\end{align*}
For the Navier-Stokes Boussinesq system, when $\g^2>\f14$, the asymptotic behavior of $P_{\neq}\om$ changes. The vorticity has an additional transicent growth $\sqrt{t}$, namely,
\begin{align*}
\|P_{\neq}\om\|_{L^2}\leq C\sqrt{\langle t\rangle}e^{-c\nu^{\f13}t}.
\end{align*}
To stabilize this growth, we need an additial $\nu^{\f16}$ smallness, due to the fact that $\sqrt{\langle t\rangle}e^{-c\nu^{\f13}t}\lesssim \nu^{-\f16}$. Thus it is somehow necessary to prove the asymptotic stability under the assumption that the initial perturbation $\om_{in}$ is of size $\nu^{\f13}\nu^{\f16}=\nu^{\f12}$. 

The size of the non-zero mode of the initial perturbations  $(P_{\neq}u_{in}, P_{\neq}\th_{in})$ seems optimal. See Section \ref{optimality} for a formal discussion of the sharpness of the size. We also note that the size of the zero mode of the initial perturbations $(P_0u_{in}, P_0\th_{in})$ can be slightly larger. 
\end{remark}
\subsection{Notations and conventions}
By convention, we always use Greek letters such as $\eta$ and $\xi$ to denote frequencies in the $y$ direction and lowercase Latin characters commonly used as indices such as $k$ and $l$ to denote frequencies in the $x$ or $z$ direction (which are discrete). 

For a Schwartz function $f=f(z,y):\bbT\times\bbR\to\bbR$, we define the Fourier transform as 
\begin{align*}
\wh{f}(k,\eta)=\f{1}{2\pi}\int_{\bbT\times\bbR}f(z,y)e^{-ikz-iy\eta}dzdy,
\end{align*}
and the Fourier inversion of $\wh{f}(k,\eta)$ is 
\begin{align*}
\cF^{-1}\wh{f}(z,y)=\f{1}{2\pi}\sum_{k\in\bbZ}\int_{\bbR}\wh{f}(k,\eta)e^{ikz+iy\eta}d\eta.
\end{align*}
With these definitions, we have $\wh{fg}=\wh{f}\ast\wh{g}$ and
\begin{align*}
\int f(z,y)g(z,v)dzdv=\sum_{k\in\bbZ}\int_{\bbR}\wh{f}(k,\eta)\wh{g}(k,\eta)d\eta.
\end{align*}

We use the notation $f\lesssim g$ when there exists a constant $C>0$ independent of the parameters of interest such that $f\leq Cg$ (we define $g\gtrsim f$ analogously). Similarly, we use the notation $f\approx g$ when there exists $C>0$ such that $C^{-1}g\leq f\leq Cg$.

We will denote the $l^2$ vector norm $|k,\eta|=\big(|k|^2+|\eta|^2\big)^{\f12}$, which by convention is the norm taken in our work. Similarly, given a scalar or vector in $\bbR^n$, we denote
\begin{align*}
\langle v\rangle=(1+|v|^2)^{\f12}.
\end{align*}

For any function $f$ defined on $\bbR$, we denote its Sobolev norm by 
\begin{align*}
\|f\|_{H^s}=\|\langle D_y\rangle^sf\|_{L^2}=\|\langle \eta\rangle^s\wh{f}\|_{L^2}.
\end{align*}

For any function $f$ defined on $\bbT\times\bbR$, we denote its Sobolev norm by 
\begin{align*}
\|f\|_{H^s}=\|\langle \na\rangle^sf\|_{L^2}=\Big(\sum_{k\in\bbZ}\int_{\bbR}\langle k,\eta\rangle^{2s}|\wh{f}(k,\eta)|^2d\eta\Big)^{\f12}.
\end{align*}

We denote the projection to the zero mode by
\begin{align*}
P_0f(y)=f_0(y)=\f{1}{2\pi}\int_{\bbT}f(z,y)dz,
\end{align*}
and denote the projection to the non-zero mode by 
\begin{align*}
P_{\neq}f(z,y)=f_{\neq}(z,y)=f(z,y)-P_0f(y)=f(z,y)-f_0(y).
\end{align*}


For a statement $Q$, $\mathds{1}_{Q}$ will denote the function that equals 1 if $Q$ is true and 0 otherwise.

\section{Main idea and the proof of Theorem \ref{thm}}
In this section, we give the main idea and sketch of the proof of Theorem \ref{thm}.
\subsection{Change of coordinates}
First, we introduce the change of coordinates:
\begin{align}\label{coor change}
&z=x-ty, \quad f(t,z,y)=w(t,x,y),\quad \th(t,z,y)=\vth(t,x,y),\nonumber\\
&\wt{u}(t,z,y)=u(t,x,y), \quad \phi(t,z,y)=\psi(t,x,y).
\end{align}

Under the change of coordinates, we deduce the following nonlinear equations from \eqref{eq: vorticity},
\beq\label{eq: coor nonlinear}
 \left\{\begin{array}{l}
\pa_tf+\wt{u}\cdot\na_Lf-\nu\D_L f=-\g^2\pa_{z}\th,\\
\pa_t\th+\wt{u}\cdot\na_L\th-\nu\D_L\th=\wt{u}^2=-\pa_z(-\D_L)^{-1}f,\\
\wt{u}=\na_L^{\perp}\phi,\quad \D_L\phi=f.
\end{array}\right.
\eeq

We  divide the solution $(f,\th)$ into the zero mode and non-zero modes. Let $f=f_0+f_{\neq}$, whereas $f_0$ satisfies
\beq\label{eq: f0}
 \left\{\begin{array}{l}
\pa_tf_0-\nu\pa_{yy}f_0=-\pa_y(\wt{u}_{\neq}^2f_{\neq})_0,\\
f_0|_{t=0}=P_0f_{in}.
\end{array}\right.
\eeq
Let $\th=\th_{01}+\th_{02}+\th_{\neq}$, whereas $\th_{01}$ satisfies
\beq\label{eq: th01}
 \left\{\begin{array}{l}
\pa_t\th_{01}-\nu\pa_{yy}\th_{01}=0,\\
\th_{01}|_{t=0}=P_0\th_{in},
\end{array}\right.
\eeq
and $\th_{02}$ satisfies
\beq\label{eq: th02}
 \left\{\begin{array}{l}
\pa_t\th_{02}-\nu\pa_{yy}\th_{02}+(\wt{u}_{\neq}\cdot\na_L\th_{\neq})_0=0,\\
\th_{02}|_{t=0}=0.
\end{array}\right.
\eeq
\begin{remark}
The main idea of the decomposition of the zero mode of $\th$, namely $\th_0=\th_{01}+\th_{02}$, is to devide $\th_0$ into two part, the first part $\th_{01}$ has higher regularity with larger size, the second part $\th_{02}$ has lower regularity with smaller size. 
\end{remark}
We also obtain that for $(f_{\neq},\th_{\neq})$, 
\beq\label{eq:f,thnonzero}
 \left\{\begin{array}{l}
\pa_tf_{\neq}+(\wt{u}\cdot\na_Lf)_{\neq}-\nu\D_L f_{\neq}=-\g^2\pa_{z}\th_{\neq},\\
\pa_t\th_{\neq}+(\wt{u}\cdot\na_L\th)_{\neq}-\nu\D_L\th_{\neq}=\wt{u}^2_{\neq},\\
(f_{\neq},\th_{\neq})|_{t=0}=(P_{\neq}f_{in},P_{\neq}\th_{in}).
\end{array}\right.
\eeq

We also need to estimate the zero mode $\wt{u}_0^1$ when estimating the nonlinear terms of the non-zero modes system. From \eqref{eq: u sys},  it is easy to get that $\wt{u}_0^1$ satisfies
\beq\label{eq: u0}
 \left\{\begin{array}{l}
\pa_t\wt{u}_0^1-\nu\pa_{yy}\wt{u}_0^1=-\pa_y(\wt{u}^2_{\neq}\wt{u}_{\neq}^1)_0,\\
\wt{u}_0^1|_{t=0}=P_0\wt{u}^1_{in}.
\end{array}\right.
\eeq

By using the standard energy estimate, it is easy to get that from \eqref{eq: th01},
\begin{lemma}
It holds that 
\begin{align*}
\f{d}{dt}\|\th_{01}\|_{H^{s+2}}^2+2\nu\|\pa_y\th_{01}\|_{H^{s+2}}^2=0.
\end{align*}
\end{lemma}

\subsection{Linearized system and good unknowns. }
Before beginning the proof of Theorem \ref{thm}, we first discuss the corresponding linearized system:
\begin{equation}\label{Euler-Boussinesq}
 \left\{\begin{array}{l}
\pa_t\om+y\pa_x\om-\nu\D \om=-\g^2\pa_x\vth,\\
\pa_t\vth+y\pa_x\vth-\nu\D\vth=u^2,\\
u=(u^1,u^2)=(-\pa_y\psi,\pa_x\psi),\quad \D\psi=\om.
\end{array}\right.
\end{equation}

Under the new coordinates \eqref{coor change},  we get that $(f,\th)$ satisfies
\beq\label{eq: new coor}
 \left\{\begin{array}{l}
\pa_tf-\nu\D_L f=-\g^2\pa_{z}\th,\\
\pa_t\th-\nu\D_L\th=\wt{u}^2,\\
\wt{u}=\na_L^{\perp}\phi,\quad \D_L\phi=f,
\end{array}\right.
\eeq
where $\na_L=\Big(\begin{array}{ccc}
   \pa_{z}\\
   \pa_y-t\pa_{z}\\
  \end{array}\Big)$ and $\D_L=\pa_{z}^2+(\pa_y-t\pa_{z})^2$.
  
For \eqref{eq: new coor}, by taking Fourier transform in $z$ and $y$, we obtain
\beq\label{eq: fourier1}
 \left\{\begin{array}{l}
\pa_t\wh{f}+\nu(k^2+(\eta-kt)^2)\wh{f}=-\g^2 ik\wh{\th},\\
\pa_t\wh{\th}+\nu(k^2+(\eta-kt)^2)\wh{\th}=-\f{ik}{k^2+(\eta-kt)^2}\wh{f},\\
\wh{\wt{u}}=\Big(\begin{array}{ccc}
   -i(\eta-kt)\wh{\phi}\\
   ik\wh{\phi}\\
  \end{array}\Big),\quad -(k^2+(\eta-kt)^2)\wh{\phi}=\wh{f}.
\end{array}\right.
\eeq

For the sake of presentation, we obtain that by denoting $\wh{\Th}=ik\wh{\th}$,
\beq\label{eq: W Th}
 \left\{\begin{array}{l}
\pa_t\wh{f}+\nu (k^2+(\eta-kt)^2)\wh{f}=-\g^2 \wh{\Th},\\
\pa_t\wh{\Th}+\nu (k^2+(\eta-kt)^2)\wh{\Th}=\f{k^2}{(k^2+(\eta-kt)^2)}\wh{f}.
\end{array}\right.
\eeq

For our purposes, it is more convenient to recall the energy method used in \cite{Bianchini}, originally introduced to deal with the linear stability of the Couette flow in a compressible fluid \cite{ADM, ZZZi}. Indeed, we have that for $k\neq0$ and $t\geq\f{\eta}{k}$,
\begin{align*}
&\f{d}{dt}\left(\big(1+(t-\f{\eta}{k})^2\big)^{-\f12}|\wh{f}|^2+\big(1+(t-\f{\eta}{k})^2\big)^{\f12}\g^2|\wh{\Th}|^2+\wh{f}\ol{\wh{\Th}}\right)\\
&\quad+2\nu \big(k^2+(\eta-kt)^2\big)\left(\big(1+(t-\f{\eta}{k})^2\big)^{-\f12}|\wh{f}|^2+\big(1+(t-\f{\eta}{k})^2\big)^{\f12}\g^2|\Th|^2+\wh{f}\ol{\wh{\Th}}\right)\nonumber\\
&\lesssim \big(1+(t-\f{\eta}{k})^2\big)^{-\f32}\left(\big(1+(t-\f{\eta}{k})^2\big)^{-\f12}|\wh{f}|^2+\big(1+(t-\f{\eta}{k})^2\big)^{\f12}\g^2|\wh{\Th}|^2\right),
\end{align*}
and for $t\leq\f{\eta}{k}$,
\begin{align*}
&\f{d}{dt}\left(\big(1+(t-\f{\eta}{k})^2\big)^{-\f12}|\wh{f}|^2+\big(1+(t-\f{\eta}{k})^2\big)^{\f12}\g^2|\wh{\Th}|^2-\wh{f}\ol{\wh{\Th}}\right)\\
&\quad \quad+2\nu \big(k^2+(\eta-kt)^2\big)\left(\big(1+(t-\f{\eta}{k})^2\big)^{-\f12}|\wh{f}|^2+\big(1+(t-\f{\eta}{k})^2\big)^{\f12}\g^2|\Th|^2-\wh{f}\ol{\wh{\Th}}\right)\nonumber\\
&\quad\lesssim \big(1+(t-\f{\eta}{k})^2\big)^{-\f32}\left(\big(1+(t-\f{\eta}{k})^2\big)^{-\f12}|\wh{f}|^2+\big(1+(t-\f{\eta}{k})^2\big)^{\f12}\g^2|\wh{\Th}|^2\right),
\end{align*}
which implies that 
\begin{align}\label{eq:linear-est}
\begin{split}
|\widehat{f}(t,k,\eta)|&\lesssim \f{|k^2+(kt-\eta)^2|^{\f14}}{|k^2+\eta^2|^{\f14}}|\wh{f}_{in}(k,\eta)|e^{-c\nu k^2t^3},\\
|\widehat{\th}(t,k,\eta)|&\lesssim \f{|k^2+\eta^2|^{\f14}}{|k^2+(kt-\eta)^2|^{\f14}}|\wh{\th}_{in}(k,\eta)|e^{-c\nu k^2t^3}.
\end{split}
\end{align}

To attack the nonlinear problem, the linear estimate \eqref{eq:linear-est} is not enough, thus we introduce the good unknowns $(X_1, X_2)$ which enjoy a better system. The idea of finding the good unknowns is to symmetrize the system \eqref{Euler-Boussinesq} via time-dependent Fourier multipliers.
First, let us introduce two time-dependent Fourier multipliers $\mathcal{N}=\mathcal{N}(\na_{L})$ and $\dot{\mathcal{N}}=\dot{\mathcal{N}}(\na_{L})$, which are defined by
\beno
\mathcal{N}(\na_{L})f_{\neq}\eqdef|\pa_z|^{-\f12}(-\Delta_L)^{\f14}f_{\neq}=\mathcal{F}^{-1}\Big(|k|^{-\f12}(k^2+(\eta-kt)^2)^{\f14}\hat{f}_{\neq}(k,\eta)\Big),
\eeno
and
\begin{align*}
\dot{\mathcal{N}}(\na_{L})f_{\neq}&\eqdef-\f{1}{2}(\pa_y-t\pa_z)\pa_z|\pa_z|^{-\f12}(-\Delta_L)^{-\f34}f_{\neq}\\
&=\mathcal{F}^{-1}\Big(\f12(\eta-kt)k|k|^{-\f12}(k^2+(\eta-kt)^2)^{-\f34}\hat{f}_{\neq}(k,\eta)\Big).
\end{align*}
Next, let us introduce the good unknowns
\begin{align*}
X_1\eqdef\mathcal{N}^{-1}f_{\neq},\quad X_2\eqdef(1-\f{1}{4\g^2})^{-\f12}\left(\g^{-1}\dot{\mathcal{N}}f_{\neq}+\g\mathcal{N}\pa_z\th_{\neq}\right),
\end{align*}
and it is easy to obtain that 
\beq\label{def f0 th0}
f_{\neq}=\mathcal{N} X_1,\quad \th_{\neq}=\g^{-1}(\pa_z)^{-1}\left((1-\f{1}{4\g^2})^{\f12}\mathcal{N}^{-1}X_2-\f12\g^{-1}\dot{\mathcal{N}}X_1\right).
\eeq

Thus, the nonlinear system \eqref{eq:f,thnonzero} can be rewritten as follows:
\beq\label{eq: X1X2}
 \left\{\begin{array}{l}
\pa_tX_1-\nu\D_L X_1=-\f{\s}{2}(\cN^{-1})^2X_2-\cN^{-1}(\wt{u}\cdot\na_L f)_{\neq},\\
\pa_tX_2-\nu\D_L X_2=\f{\s}{2}(\cN^{-1})^2X_1+\f{3}{2}\s^{-1}\pa_z^2(-\D_L)^{-1}\cN^{-2}X_1\\
\quad-2\g^2\s^{-1}\pa_z\cN(\wt{u}\cdot\na_L\th)_{\neq}-\s^{-1}\dot{\cN}(\wt{u}\cdot\na_L f)_{\neq}.
\end{array}\right.
\eeq
Here and as follows, we will denote $\s=\sqrt{4\g^2-1}$ for brevity. 

In the following, we focus on the equations \eqref{eq: f0}, \eqref{eq: th02}, \eqref{eq: u0} and \eqref{eq: X1X2}.
\subsection{Main multipliers}
We use  some key time-dependent Fourier multipliers $\mathcal{N}$, $\dot{\mathcal{N}}$, $\cA$ and $m$, whereas $\mathcal{N}$ and $\dot{\mathcal{N}}$ are used to capture the linear growth/decay of the nonzero modes of $(f,\th)$.  To control the growth from the nonlinear interactions, we introduce $\cA$ and $m$.
\subsubsection{Multipliers $\cA$} We use the Fourier multiplier operator $\cA$ to control the growth of non-zero modes, which is defined as:
\begin{align*}
\cA f(t,z,y)\eqdef\cF^{-1}\Big(\cA_k^s(t,\eta)\wh{f}(t,k,\eta)\Big),
\end{align*}
with
 \begin{align*}
 \cA_k^s(t,\eta)\eqdef e^{c\nu^{\f13}t}\cM(t,k,\eta)\langle k,\eta\rangle^s.
 \end{align*}
Here, the main multiplier $\cM$ is defined by 
\begin{align*}
\cM\eqdef\exp\{K(\cM_1+\cM_2+\cM_3)\},
\end{align*}
where $K$ is a large constant, and $\cM_i, i=1,2,3$ are defined by 
 \begin{align*}
\pa_t\cM_1&=-\nu^{\f13}|k|^{\f23}\va'\big(\nu^{\f13}|k|^{-\f13}\text{sgn}(k)(\eta-kt)\big),\\
\pa_t\cM_2&=-\f{k^2}{k^2+(\eta-kt)^2},\\
\pa_t\cM_3&=-\f{1}{(k^2+(\eta-kt)^2)^{\f34}},
\end{align*}
with the initial data $\cM_i|_{t=0}=0$ for $i=1,2,3,$ and here $\va\in C^{\infty}(\bbR)$ is a real-valued, non-decreasing function such that  $0\leq\va\leq1$, $\va'\geq0$ and $\va'=\f14$ on $[-1,1]$. Thus, we can easily get 
 \begin{align}\label{M bound}
 0<c_0(K)\leq\cM\leq1.
 \end{align}
 
 As a consequence, 
 \ben\label{eq:A bound}
 \cA_k^s(t,\eta)\approx e^{c\nu^{\f13}t}\langle k,\eta\rangle^s.
 \een 
 
 The operator $\cM_1$ is used to obtain the enhanced dissipation. Indeed, we have the following lemma: 
\begin{lemma}\label{low bound}
It holds for all $t, \eta$ and $k\neq0$ that 
\begin{align*}
\f14\nu(k^2+(\eta-kt)^2)+\nu^{\f13}|k|^{\f23}\va'\big(\nu^{\f13}|k|^{-\f13}\mathrm{sgn}(k)(\eta-kt)\big)\geq \f{1}{4}\nu^{\f13}|k|^{\f23}.
\end{align*}
\end{lemma}
\begin{proof}
For the case $|\eta-kt|\leq \nu^{-\f13}|k|^{\f13}$, by the definition of $\varphi$, we have 
$$\nu^{\f13}|k|^{\f23}\va'\big(\nu^{\f13}|k|^{-\f13}\mathrm{sgn}(k)(\eta-kt)\big)\geq \f14\nu^{\f13}|k|^{\f23}.$$
 For the case $|\eta-kt|\geq \nu^{-\f13}|k|^{\f13}$, it holds that $\f14\nu(k^2+(\eta-kt)^2)\geq \f{1}{4}\nu^{\f13}|k|^{\f23}$. Thus we proved the lemma. 
\end{proof}
The operator $\cM_2$ is to obtain the inviscid damping. For $|k|$ small, the operator $\cM_3$ is slightly stronger than $\cM_2$ in terms of the time decay. We shall also use the following commutator estimates.
\begin{lemma}\label{M lem}
There holds that for $k\neq0$,
\begin{align*}
&\cM(t,k,\eta)\langle k,\eta\rangle^s-\cM(t,k,\xi)\langle k, \xi\rangle^s\\
&\quad\lesssim|\eta-\xi|\big(\nu^{\f13}|k|^{-\f13}+|k|^{-1}\big)\langle k,\xi\rangle^s+|\eta-\xi|\langle k,\eta-\xi\rangle^{s-1}.
\end{align*}
\end{lemma}
\begin{proof}
By the definition of $\cM$, we have
\begin{align}\label{a}
&\cM(t,k,\eta)-\cM(t,k,\xi)\nonumber\\
&=\Big[(e^{K\cM_1})(t,k,\eta)-(e^{K\cM_1})(t,k,\xi)\Big](e^{K\cM_2+K\cM_3})(t,k,\eta)\nonumber\\
&\quad+(e^{K\cM_1})(t,k,\xi)(e^{K\cM_3})(t,k,\eta)\Big[(e^{K\cM_2})(t,k,\eta)-(e^{K\cM_2})(t,k,\xi)\Big]\nonumber\\
&\quad+(e^{K\cM_1})(t,k,\xi)(e^{K\cM_2})(t,k,\xi)\Big[(e^{K\cM_3})(t,k,\eta)-(e^{K\cM_3})(t,k,\xi)\Big].
\end{align}
According to the definition of $\cM_1$, we have
\begin{align*}
\pa_{\xi}\cM_1(t,k,\xi)=\nu^{\f13}|k|^{-\f13}\text{sgn}(k)\va'\big(\nu^{\f13}|k|^{-\f13}\text{sgn}(k)(\xi-kt)\big),
\end{align*}
and then by \eqref{M bound}, we obtain
\begin{align}\label{a1}
&|(e^{K\cM_1})(t,k,\eta)-(e^{K\cM_1})(t,k,\xi)|\nonumber\\
&=\big|(\eta-\xi)\int_0^1(e^{K\cM_1})(t,k,\xi)K\pa_{\xi}\cM_1(t,k,\xi)d\xi\Big|\nonumber\\
&=\Big|K(\eta-\xi)\int_0^1(e^{K\cM_1})(t,k,\xi)\nu^{\f13}|k|^{-\f13}\text{sgn}(k)\va'\big(\nu^{\f13}|k|^{-\f13}\text{sgn}(k)(\xi-kt)\big)d\xi\Big|\nonumber\\
&\lesssim |\eta-\xi|\nu^{\f13}|k|^{-\f13}.
\end{align}
By the definition of $\cM_2$, we have 
\begin{align*}
\cM_2(t,k,\eta)-\cM_2(t,k,\xi)&=-\int_0^t\f{k^2}{k^2+(\eta-k\tau)^2}d\tau+\int_0^t\f{k^2}{k^2+(\xi-k\tau)^2}d\tau\\
&=\int_{\f{\eta}{k}}^{\f{\eta}{k}-t}\f{1}{1+z^2}dz-\int_{\f{\xi}{k}}^{\f{\xi}{k}-t}\f{1}{1+z^2}dz\\
&=\int_{\f{\eta}{k}}^{\f{\xi}{k}}\f{1}{1+z^2}dz-\int_{\f{\eta}{k}-t}^{\f{\xi}{k}-t}\f{1}{1+z^2}dz,
\end{align*}
and then we get 
\begin{align*}
|\cM_2(t,k,\eta)-\cM_2(t,k,\xi)|\lesssim \f{|\eta-\xi|}{|k|}.
\end{align*}
Thus, we obtain that by \eqref{M bound},
\begin{align}\label{a2}
|(e^{K\cM_2})(t,k,\eta)-(e^{K\cM_2})(t,k,\xi)|
&\lesssim |\cM_2(t,k,\eta)-\cM_2(t,k,\xi)|\lesssim \f{|\eta-\xi|}{|k|}.
\end{align}

By the similar argument, we get that 
\begin{align*}
|\cM_3(t,k,\eta)-\cM_3(t,k,\xi)|\lesssim \f{|\eta-\xi|}{|k|^{\f52}},
\end{align*}
and 
\begin{align}\label{a3}
&|(e^{K\cM_3})(t,k,\eta)-(e^{K\cM_3})(t,k,\xi)|\lesssim \f{|\eta-\xi|}{|k|^{\f52}}.
\end{align}

Then,  from \eqref{a}, we obtain that by combining \eqref{M bound}, \eqref{a1}, \eqref{a2} and \eqref{a3},
\begin{align*}
|\cM(t,k,\eta)-\cM(t,k,\xi)|\lesssim |\eta-\xi|\Big(\nu^{\f13}|k|^{-\f13}+|k|^{-1}\Big).
\end{align*}

Thus, for $k\neq0$, we obtain
\begin{align*}
&\cM(t,k,\eta)\langle k,\eta\rangle^s-\cM(t,k,\xi)\langle k, \xi\rangle^s\\
&=\big[\cM(t,k,\eta)-\cM(t,k,\xi)\big]\langle k,\xi\rangle^s+\cM(t,k,\eta)\big[\langle k,\eta\rangle^s-\langle k,\xi\rangle^s\big]\nonumber\\
&\lesssim|\eta-\xi|\Big(\nu^{\f13}|k|^{-\f13}+|k|^{-1}\Big)\langle k,\xi\rangle^s+|\eta-\xi|\big[\langle k,\eta\rangle^{s-1}+\langle k,\xi\rangle^{s-1}\big]\nonumber\\
&\lesssim|\eta-\xi|\Big(\nu^{\f13}|k|^{-\f13}+|k|^{-1}\Big)\langle k,\xi\rangle^s+|\eta-\xi|\langle k,\eta-\xi\rangle^{s-1}.
\end{align*}

We complete the proof of the lemma.
\end{proof}
\subsubsection{Multipliers $m$}\label{multiplier m} We use the multiplier operator $m$ to control the nonlinear growth of zero-mode $\th_{02}$, see Section 5.3 for more details. 

For $|\eta|\geq3$, let 
\beno
t(\eta)\eqdef\f{2|\eta|}{2E(\sqrt{|\eta|})+1},
\eeno
where $E(\sqrt{|\eta|})$ is  the integer part of $\sqrt{|\eta|}$.
For $|k|=1,2,3,\cdots,E(\sqrt{|\eta|})$ and $k\eta>0$, we denote the resonance interval by
$$I_{k,\eta}=\Big[\f{2\eta}{2k+1},\f{2\eta}{2k-1}\Big].$$
Then $t(\eta)\approx\sqrt{|\eta|}$ and 
$$[t(\eta),2|\eta|]=\bigcup_{k=1}^{E(\sqrt{|\eta|})}I_{k,\eta}.$$

For $|\eta|<3$, we define $m(t,\eta)\equiv1$.

For $|\eta|\geq3$, if $t\geq 2|\eta|$, we define $m(t,\eta)\equiv1$; \\
if $t\in[\f23\eta,2\eta]=I_{1,\eta}$, we define $m(t,\eta)$ by solving
\beno
 \left\{\begin{array}{l}
\pa_tm=-\f{m}{(1+(t-\eta)^2)^{\f34}},\\
m(2\eta,\eta)=1;\\
\end{array}\right.
\eeno
if $t\in I_{k,\eta}$ with $k=2,3,\cdots, E(\sqrt{|\eta|})$, we define $m(t,\eta)$ by solving 
\beno
 \left\{\begin{array}{l}
\pa_tm=-\f{m}{k^2(1+(t-\f{\eta}{k})^2)^{\f34}},\\
m|_{t=\f{2\eta}{2k-1}}=m(\f{2\eta}{2k-1},\eta);\\
\end{array}\right.
\eeno
if $t\leq t(\eta)$, let $m(t,\eta)=m\Big(\f{2\eta}{2E(\sqrt{|\eta|})+1},\eta\Big)$.

\begin{lemma}
It holds that 
\begin{align*}
m(t,\eta)\approx 1.
\end{align*}
\end{lemma}
\begin{proof}
For any $t,\eta$, it is easy to see that $m(t,\eta)\leq1$ and
\begin{align*}
\f{1}{m\Big(\f{2\eta}{2E(\sqrt{|\eta|})+1},\eta\Big)}&=\prod_{k=1}^{E(\sqrt{|\eta|})}\exp\left\{\int_{\f{2\eta}{2k+1}}^{\f{2\eta}{2k-1}}\f{1}{k^2}\Big(1+(t-\f{\eta}{k})^2\Big)^{-\f34}dt\right\}\\
&\lesssim \prod_{k=1}^{E(\sqrt{|\eta|})}\exp\Big\{\f{1}{k^2}\Big\}\lesssim 1.
\end{align*}
Thus we have proved the lemma.
\end{proof}
For any function $f(t,y)$ defined on $\mathbb{R}^+\times \mathbb{R}$, we define
\beno
m^{-1}f=m(t,\pa_y)^{-1}f\eqdef \mathcal{F}^{-1}\left(\f{\widehat{f}(t,\eta)}{m(t,\eta)}\right),
\eeno
and
\beno
\sqrt{\f{\pa_tm}{m}}f=\sqrt{\f{\pa_tm(t,\pa_y)}{m(t,\pa_y)}}f\eqdef\mathcal{F}^{-1}\left(\sqrt{\f{\pa_tm(t,\eta)}{m(t,\eta)}}f(t,\eta)\right).
\eeno

\subsection{Main energy estimates}


We define
\begin{align}
E_{\neq}\eqdef\f12\Big(\|\cA X_1\|_{L^2}^2+\|\cA X_2\|_{L^2}^2\Big),
\end{align}
\begin{align}
F_0\eqdef\|f_0\|_{H^s}^2,\quad H_{02}\eqdef\left\|m^{-1}\th_{02}\right\|_{H^s}^2,\quad 
V_0\eqdef\|\wt{u}_0^1\|_{H^s}^2.
\end{align}
From the time evolution of $E_{\neq}$, we get
\begin{align}\label{eq: E neq1}
&\f12\f{d}{dt}\Big(\|\cA X_1\|_{L^2}^2+\|\cA X_2\|_{L^2}^2\Big)+\nu\Big(\|\na_L\cA X_1\|_{L^2}^2+\|\na_L\cA X_2\|_{L^2}^2\Big)+\mathfrak{CK}(t)\nonumber\\
&=c\nu^{\f13}\Big(\|\cA X_1\|_{L^2}^2+\|\cA X_2\|_{L^2}^2\Big)
+\f{3}{2\s}\int\cA\pa_z^2(-\D_L)^{-1}\cN^{-2}X_1\cA X_2dzdy\\
&\quad-\int\cA\cN^{-1}(\wt{u}\cdot\na_L f)_{\neq}\cA X_1dzdy\nonumber\\
&\quad+\int\cA\Big(-2\g^2\s^{-1}\pa_z\cN(\wt{u}\cdot\na_L\th)_{\neq}-\s^{-1}\dot{\cN}(\wt{u}\cdot\na_L f)_{\neq}\Big)\cA X_2dzdy\nonumber\\
&\eqdef c\nu^{\f13}\|\cA X_1\|_{L^2}^2+c\nu^{\f13}\|\cA X_2\|_{L^2}^2+L+NL_1+NL_2,
\end{align}
whereas we used the fact that 
\begin{align*}
-\f{\s}{2}\int\cA(\cN^{-1})^2X_2\cA X_1dzdy+\f{\s}{2}\int\cA(\cN^{-1})^2X_1\cA X_2dzdy=0.
\end{align*}
Here in \eqref{eq: E neq1}, $L$ is the linear term, $NL_1, NL_2$ are the nonlinear terms and $\mathfrak{CK}$ stands for "Cauchy-Kovalevskaya",
\begin{align*}
\mathfrak{CK}(t)&\eqdef K\sum_{k\neq0}\int-\pa_t\cM_1(t,k,\eta)\Big(|\cA_k^s(t,\eta)\wh{X}_1(t,k,\eta)|^2+|\cA_k^s(t,\eta)\wh{X}_2(t,k,\eta)|^2\Big)d\eta\\
&\quad+K\sum_{k\neq0}\int-\pa_t\cM_2(t,k,\eta)\Big(|\cA_k^s(t,\eta)\wh{X}_1(t,k,\eta)|^2+|\cA_k^s(t,\eta)\wh{X}_2(t,k,\eta)|^2\Big)d\eta\\
&\quad+K\sum_{k\neq0}\int-\pa_t\cM_3(t,k,\eta)\Big(|\cA_k^s(t,\eta)\wh{X}_1(t,k,\eta)|^2+|\cA_k^s(t,\eta)\wh{X}_2(t,k,\eta)|^2\Big)d\eta\\
&\eqdef \sum_{j=1}^3\mathfrak{CK}_j(t).
\end{align*}

By Lemma \ref{low bound}, we get
\begin{align*}
\f14\nu\Big(\|\na_L\cA X_1\|_{L^2}^2+\|\na_L\cA X_2\|_{L^2}^2\Big)+\mathfrak{CK}_1&\geq \f14\nu^{\f13}\Big(\||D_z|^{\f13}\cA X_1\|_{L^2}^2+\||D_z|^{\f13}\cA X_2\|_{L^2}^2\Big)\\
&\eqdef 2\mathfrak{ED},
\end{align*}
where $\mathfrak{ED}$ stands for the enhanced dissipation.

On the other hand, by the definition of the multipliers $\cM_2$ and $\cM_3$, we have
\begin{align*}
\mathfrak{CK}_2&=K\Big(\|\pa_z(-\D_L)^{-\f12}\cA X_1\|_{L^2}^2+ \|\pa_z(-\D_L)^{-\f12}\cA X_2\|_{L^2}^2\Big),\\
\mathfrak{CK}_3&=K\Big(\|(-\D_L)^{-\f38}\cA X_1\|_{L^2}^2+\|(-\D_L)^{-\f38}\cA X_2\|_{L^2}^2\Big).
\end{align*}

For the linear term $L$,  by taking $K\geq3\s^{-1}$, we get that 
\begin{align*}
&\f{3}{2\s}\int\cA\big(\pa_z^2(-\D_L)^{-1}\cN^{-2}X_1\big)\cA X_2dzdy\\
&\leq \f{3}{2\s}\|\pa_z(-\D_L)^{-\f12}\cA X_1\|_{L^2}\|\pa_z(-\D_L)^{-\f12}\cA X_2\|_{L^2}\\
&\leq\f14K\|\pa_z(-\D_L)^{-\f12}\cA X_1\|_{L^2}^2+\f14K\|\pa_z(-\D_L)^{-\f12}\cA X_2\|_{L^2}^2.
\end{align*}
Thus, by choosing $0<c\leq\f18$ and denoting $\mathfrak{D}$ be the diffusion:
\begin{align*}
\mathfrak{D}\eqdef\f34\nu\Big(\|\na_L\cA X_1\|_{L^2}^2+\|\na_L\cA X_2\|_{L^2}^2\Big),
\end{align*}
we obtain that from \eqref{eq: E neq1},
\begin{align}\label{eq: E neq2}
&\f{d}{dt}E_{\neq}+\mathfrak{D}+\mathfrak{ED}+\mathfrak{CK}_2+\mathfrak{CK}_3\nonumber\\
&\leq\Big|\int\cA\cN^{-1}(\wt{u}\cdot\na_L f)_{\neq}\cA X_1dzdy\Big|\nonumber\\
&\quad+\Big|\int\cA\big(-2\g^2\s^{-1}\pa_z\cN(\wt{u}\cdot\na_L\th)_{\neq}-\s^{-1}\dot{\cN}(\wt{u}\cdot\na_L f)_{\neq}\big)\cA X_2dzdy\Big|\nonumber\\
&=|NL_1|+|NL_2|.
\end{align}

We define the following control as the bootstrap hypotheses for $t\geq0$.
\begin{align}\label{boot1}
E_{\neq}(t)+\f12\int_0^t\Big(\mathfrak{D}+\mathfrak{ED}+\mathfrak{CK}_2+\mathfrak{CK}_3\Big)(t')dt'\leq (4\ep\nu^{\f12})^2,
\end{align}
\begin{align}
\|f_0\|_{H^s}^2+\nu\int_0^t\|\pa_yf_0(t')\|_{H^s}^2dt'\leq (4\ep\nu^{\f14})^2,
\end{align}
\begin{align}\label{boot u0}
\|\wt{u}_0^1\|_{H^s}^2+\nu\int_0^t\|\pa_y\wt{u}_0^1(t')\|_{H^s}^2dt'\leq (4\ep\nu^{\f12})^2,
\end{align}
\begin{align}\label{boot th}
\left\|m^{-1}\th_{02}\right\|_{H^s}^2+\nu\int_0^t\left\|\pa_y\Big(m^{-1}\th_{02}\Big)(t')\right\|_{H^s}^2dt'+\int_0^t\left\|\sqrt{\f{\pa_tm}{m}}m^{-1}\th_{02}(t')\right\|_{H^s}^2dt'\leq (4\ep\nu^{\f34})^2.
\end{align}

Let $I^{\ast}$ be the connected set of time $t\geq0$ such that the bootstrap hypotheses \eqref{boot1}-\eqref{boot th} are all satisfied. We will work on the regularized solutions for which we know $E_{\neq}, F_0, H_{02}, V_0$ take values continuously in time, and hence $I^{\ast}$ is a closed interval $[0,T^{\ast}]$ with $T^{\ast}\geq0$. The bootstrap is complete if we show that $I^{\ast}$ is also open, which is the purpose of the following proposition, the proof of which constitutes the majority of this work.
\begin{proposition}\label{main propo}
Let $s\geq 6$ and $\g^2>\f14$. There exist $0<\ep_0=\ep_0(s,\g^2-\f14),\nu_0=\nu_0(s,\g^2-\f14)<1$, such that for all $0<\nu\leq\nu_0$ and $0<\ep\leq\ep_0$, such that if the bootstrap hypotheses \eqref{boot1}-\eqref{boot th} hold, then for any $t\in[0,T^{\ast}]$, we have the following properties $t\in[0,T^{\ast}]$: 
\begin{align}\label{res: E}
E_{\neq}(t)+\f12\int_0^t\Big(\mathfrak{D}+\mathfrak{ED}+\mathfrak{CK}_2+\mathfrak{CK}_3\Big)(t')dt'\leq (2\ep\nu^{\f12})^2,
\end{align}
\begin{align}\label{res: f0}
\|f_0\|_{H^s}^2+\nu\int_0^t\|\pa_yf_0(t')\|_{H^s}^2dt'\leq (2\ep\nu^{\f14})^2,
\end{align}
\begin{align}\label{res: u0}
\|\wt{u}_0^1\|_{H^s}^2+\nu\int_0^t\|\pa_y\wt{u}_0^1(t')\|_{H^s}^2dt'\leq (2\ep\nu^{\f12})^2,
\end{align}
\begin{align}\label{res: th02}
\left\|m^{-1}\th_{02}\right\|_{H^s}^2+\nu\int_0^t\left\|\pa_y\Big(m^{-1}\th_{02}\Big)(t')\right\|_{H^s}^2dt'+\int_0^t\Big\|\sqrt{\f{\pa_tm}{m}}m^{-1}\th_{02}(t')\Big\|_{H^s}^2dt'\leq (2\ep\nu^{\f34})^2,
\end{align}
from which it follows that $T^{\ast}=+\infty$.
\end{proposition}

The remainder of the paper is devoted to the proof of Proposition \ref{main propo}. 
One of the key estimates is to control the nonlinear terms $NL_1$ and $NL_2$. By using the fact that
\begin{align*}
\wt{u}\cdot\na_Lf=\wt{u}_0^1\pa_zf_{\neq}+\wt{u}_{\neq}^2\pa_yf_0+\wt{u}_{\neq}^1\pa_zf_{\neq}+\wt{u}_{\neq}^2(\pa_y-t\pa_z)f_{\neq},
\end{align*}
and 
\begin{align*}
\wt{u}\cdot\na_L\th=\wt{u}_0^1\pa_z\th_{\neq}+\wt{u}_{\neq}^2\pa_y(\th_{01}+\th_{02})+\wt{u}_{\neq}^1\pa_z\th_{\neq}+\wt{u}_{\neq}^2(\pa_y-t\pa_z)\th_{\neq},
\end{align*}
we have
\begin{align}\label{eq: NL1}
NL_1
&=-\int\cA\cN^{-1}(\wt{u}_0^1\pa_zf_{\neq})_{\neq}\cA X_1dzdy-\int\cA\cN^{-1}(\wt{u}_{\neq}^2\pa_yf_0)_{\neq}\cA X_1dzdy\nonumber\\
&\quad -\int\cA\cN^{-1}(\wt{u}_{\neq}^1\pa_zf_{\neq})_{\neq}\cA X_1dzdy-\int\cA\cN^{-1}\big(\wt{u}_{\neq}^2(\pa_y-t\pa_z)f_{\neq}\big)_{\neq}\cA X_1dzdy\nonumber\\
&\eqdef I_1+I_2+I_3+I_4,
\end{align}
and
\begin{align}\label{eq: NL2}
NL_2
&=\int\cA\Big(-2\g^2\s^{-1}\pa_z\cN(\wt{u}_0^1\pa_z\th_{\neq})_{\neq}-\s^{-1}\dot{\cN}(\wt{u}_0^1\pa_zf_{\neq}) _{\neq}\Big)\cA X_2dzdy\nonumber\\
&\quad-\f{2\g^2}{\s}\int\cA\pa_z\cN\big(\wt{u}_{\neq}^2\pa_y(\th_{01}+\th_{02})\big)_{\neq}\cA X_2dzdy-\s^{-1}\int\cA\dot{\cN}(\wt{u}_{\neq}^2\pa_yf_0)_{\neq}\cA X_2dzdy\nonumber\\
&\quad-\f{2\g^2}{\s}\int\cA\pa_z\cN(\wt{u}_{\neq}^1\pa_z\th_{\neq})_{\neq}\cA X_2dzdy-\f{2\g^2}{\s}\int\cA\pa_z\cN\big(\wt{u}_{\neq}^2(\pa_y-t\pa_z)\th_{\neq}\big)_{\neq}\cA X_2dzdy\nonumber\\
&\quad-\s^{-1}\int\cA\dot{\cN}(\wt{u}_{\neq}^1\pa_zf_{\neq})_{\neq}\cA X_2dzdy-\s^{-1}\int\cA\dot{\cN}\big(\wt{u}_{\neq}^2(\pa_y-t\pa_z)f_{\neq}\big)\cA X_2dzdy\nonumber\\
&\eqdef \sum_{i=1}^7J_i.
\end{align}

In Section \ref{sec zero-nonzero mode}, we will give the estimates of the interactions between the zero mode and the nonzero mode: $I_1,I_2,J_1,J_2,J_3$. In Section \ref{sec nonzero modes}, we will give the estimates of the interactions between non-zero modes: $I_3,I_4,J_4,J_5,J_6,J_7$. In precise, we mainly prove the following lemmas.
\begin{lemma}\label{lem: I}
Under the bootstrap hypotheses, for $t\in [0,T^{\ast}]$, it holds that 
\begin{align}
\label{eq:I_1}|I_1|&\leq C\ep\nu^{\f{1}{12}}(\mathfrak{ED})^{\f12}(\mathfrak{CK}_2)^{\f12}+C\ep\nu^{\f16}\mathfrak{ED},\\
\label{eq:I_2}|I_2|&\leq C\ep\nu^{\f14}\mathfrak{CK}_2+C\ep^3\nu^{\f54}\langle t\rangle^{-2}+C\ep\nu^{\f38}(\mathfrak{CK}_2)^{\f34}\|\pa_y\langle D_y\rangle^sf_0\|_{L^2}^{\f12}\\
\nonumber&\quad+C\ep^{\f32}\nu^{\f98}\langle t\rangle^{-\f32}
\|\pa_y\langle D_y\rangle^sf_0\|_{L^2}^{\f12},\\
\label{eq:I_3I_4}|I_3|+|I_4|&\leq C\ep\mathfrak{D}+C\ep\mathfrak{CK}_2.
\end{align}
\end{lemma}

\begin{lemma}\label{lem: J}
Under the bootstrap hypotheses, for $t\in [0,T^{\ast}]$, it holds that 
\begin{align}
\label{eq:J_1}|J_1|&\leq C\ep\nu^{\f{1}{12}}(\mathfrak{ED})^{\f12}(\mathfrak{CK}_2)^{\f12}+C\ep\nu^{\f16}(\mathfrak{ED})^{\f12}(\mathfrak{CK}_2)^{\f12}+C\ep\nu^{\f16}\mathfrak{ED},\\
\label{eq:J_2}|J_2|&\leq C\ep\nu^{\f13}(\mathfrak{ED})^{\f12}(\mathfrak{CK}_2)^{\f12}+C\ep\nu^{\f12}\mathfrak{D}^{\f14}(\mathfrak{CK}_2)^{\f14}\|\pa_y\th_{02}\|_{H^s},\\
\label{eq:J_3}|J_3|&\leq C\ep\nu^{\f14}\mathfrak{CK}_2+C\ep^3\nu^{\f54}\langle t\rangle^{-2}+C\ep\nu^{\f38}(\mathfrak{CK}_2)^{\f34}\|\pa_y\langle D_y\rangle^sf_0\|_{L^2}^{\f12}\\
\nonumber&\quad+C\ep^{\f32}\nu^{\f98}\langle t\rangle^{-\f32}
\|\pa_y\langle D_y\rangle^sf_0\|_{L^2}^{\f12},\\
\label{eq:J_4567}|J_4|+|J_5|+|J_6|+|J_7|&\leq C\ep\mathfrak{D}+C\ep\mathfrak{CK}_2.
\end{align}
\end{lemma}

In Section \ref{section zero mode}, we will prove the following proposition. 
\begin{proposition}\label{prop: zero}
Under the bootstrap hypotheses, for $t\in [0,T^{\ast}]$, there holds that  
\begin{align}\label{est: f0}
\|f_0\|_{H^s}^2+\nu\int_0^t\|\pa_yf_0(t')\|_{H^s}^2dt'\leq \|P_0f_{in}\|_{H^s}^2+C\ep^3\nu^{\f12},
\end{align}
\begin{align}\label{est: u0}
\|\wt{u}_0^1\|_{H^s}^2+\nu\int_0^t\|\pa_y\wt{u}_0^1(t')\|_{H^s}^2dt'\leq \|P_0\wt{u}^1_{in}\|_{H^s}^2+C\ep^3\nu^{\f76},
\end{align}
\begin{align}\label{est: th02}
\|m^{-1}\th_{02}\|_{H^s}^2+\nu\int_0^t\|\pa_y\Big(m^{-1}\th_{02}\Big)(t')\|_{H^s}^2dt'+\int_0^t\Big\|\sqrt{\f{\pa_tm}{m}}m^{-1}\th_{02}(t')\Big\|_{H^s}^2dt'\leq C\ep^3\nu^{\f32}.
\end{align}
\end{proposition}

\subsection{Proof of Theorem \ref{thm}} In this subsection, we first admit Lemma \ref{lem: I}, Lemma \ref{lem: J} and Proposition \ref{prop: zero} and prove Proposition \ref{main propo}. 

\begin{proof}
From Proposition \ref{prop: zero}, the estimates \eqref{res: f0}, \eqref{res: u0} and \eqref{res: th02} can be obtained directly by choosing the initial data satisfying
\begin{align*}
\|P_0f_{in}\|_{H^s}\leq \ep\nu^{\f14}, \quad \|P_0\wt{u}^1_{in}\|_{H^s}\leq \ep\nu^{\f12},
\end{align*}
and taking $\ep$ small enough.

By combining Lemma \ref{lem: I}, Lemma \ref{lem: J}, \eqref{eq: NL1}, \eqref{eq: NL2} with \eqref{eq: E neq2}, we have 
\begin{align*}
\f{d}{dt}E_{\neq}+\mathfrak{D}+\mathfrak{ED}+\mathfrak{CK}_2+\mathfrak{CK}_3
\leq& C\ep\mathfrak{ED}+C\ep\mathfrak{D}+C\ep\mathfrak{CK}_2\\
&+C\ep^3\nu^{\f54}\langle t\rangle^{-2}+C\ep\nu^{\f38}(\mathfrak{CK}_2)^{\f34}\|\pa_y\langle D_y\rangle^sf_0\|_{L^2}^{\f12}\\
&+C\ep^{\f32}\nu^{\f98}\langle t\rangle^{-\f32}
\|\pa_y\langle D_y\rangle^sf_0\|_{L^2}^{\f12}.
\end{align*}
Then by taking $\ep$ small enough,  we can directly obtain \eqref{res: E}.

Thus, we deduced the proof of Proposition \ref{main propo}.
\end{proof}

Finally, we conclude the proof of Theorem \ref{thm}.   
\begin{proof}
The estimate \eqref{zero} can be directly obtained by Proposition \ref{main propo}. By using \eqref{res: E} and the definition \eqref{def f0 th0},  we easily have \eqref{w} and \eqref{th}.
 
By using the fact that 
\begin{align*}
\wt{u}_{\neq}^1=(\pa_y-t\pa_z)(-\D_L)^{-1}f_{\neq},\quad \wt{u}_{\neq}^2=\pa_z(-\D_L)^{-1}f_{\neq}, \ \text{and}\quad f_{\neq}=|\pa_z|^{-\f12}(-\D_L)^{\f14}X_1,
\end{align*}
and 
\begin{align*}
\th_{\neq}=\g^{-1}(\pa_z)^{-1}\left((1-\f{1}{4\g^2})^{\f12}\mathcal{N}^{-1}X_2-\f12\g^{-1}\dot{\mathcal{N}}X_1\right),
\end{align*}
we get that from \eqref{res: E},
\begin{align*}
\|P_{\neq}u^1\|_{L^2}&\lesssim \||\pa_z|^{-\f12}(-\D_L)^{-\f14}X_1\|_{L^2}\leq C\ep\nu^{\f12}\langle t\rangle^{-\f12}e^{-c\nu^{\f13}t},\\
\|P_{\neq}u^2\|_{L^2}&=\|\pa_z|\pa_z|^{-\f12}(-\D_L)^{-\f34}X_1\|_{L^2}\leq C\ep\nu^{\f32}\langle t\rangle^{-\f12}e^{-c\nu^{\f13}t},\\
\|P_{\neq}f\|_{L^2}&\lesssim \||\pa_z|^{-\f12}(-\D_L)^{\f14}X_1\|_{L^2}\leq C\ep\nu^{\f12}\langle t\rangle^{\f12}e^{-c\nu^{\f13}t},\\
\|P_{\neq}\vth\|_{L^2}&\lesssim \|(\pa_z)^{-1}\mathcal{N}^{-1}X_2\|_{L^2}+\|(\pa_z)^{-1}\dot{\mathcal{N}}X_1\|_{L^2} \leq C\ep\nu^{\f12}\langle t\rangle^{-\f12}e^{-c\nu^{\f13}t}.
\end{align*}
Thus, we proved \eqref{thm: u}.

Thus, we complete the proof of Theorem \ref{thm}.
\end{proof}

\subsection{Discussion of the optimality of the size}\label{optimality} 
 In this section, we will show some evidence of the sharpness of the size $\nu^{\f12}$. For clarity, we assume that the size of $X_i, \, i=1,2$ is $\nu^{\al}$ and the size of $\th_{02}$ is $\nu^{\beta}$. 
Let us rewrite the nonlinear system as follows:
\beno
\left\{\begin{aligned}
&\pa_tX_2-\nu\Delta_LX_2=-2\g^2\s^{-1}\pa_z\cN(\wt{u}_{\neq}^2\pa_y\th_{02})+\text{good terms},\\
&\pa_t\th_{02}-\nu\pa_{yy}\th_{02}=-(\wt{u}^2_{\neq}(\pa_y-t\pa_z)\th_{\neq})_0+\text{good terms}. 
\end{aligned}
\right.
\eeno
For the term  $\pa_z\cN(\wt{u}_{\neq}^2\pa_y\th_{02})$, let us focus on the low-high interaction, namely,
\begin{align*}
\mathcal{F}\big(\pa_z\cN(\wt{u}_{\neq}^2\pa_y\th_{02})\big)&=\cF\Big(\pa_z\cN\big((\wt{u}_{\neq}^2)_{\text{low}}(\pa_y\th_{02})_{\text{high}}\big)\Big)+\text{good terms}\\
&=\sum_{k\neq 0}\int_{|\eta|>2|kt|, \, |\xi-\eta|\leq \f19|\xi|} -k|k|^{-\f12}(k^2+(\eta-kt)^2)^{\f14}\widehat{\wt{u}}^2_k(t,\eta-\xi)\xi\widehat{\th}_{02}(t,\xi)d\xi\\
&\quad +\text{good terms}. 
\end{align*}
Since $\wt{u}^2$ is the lower frequency, we regard it as $\ep\nu^{\al}\langle t\rangle^{-\f32}$. Thus formally we have
\begin{align*}
&\left|\sum_{k\neq 0}\int_{|\eta|>2|kt|, \, |\xi-\eta|\leq \f19|\xi|} -k|k|^{-\f12}(k^2+(\eta-kt)^2)^{\f14}\widehat{\wt{u}}^2(t,k,\eta-\xi)\xi\widehat{\th}_{02}(t,\xi)d\xi\right|\\
&\lesssim \sum_{k\neq 0}\int_{|\eta|>2|kt|, \, |\xi-\eta|\leq \f19|\xi|} \nu^{\al}\langle t\rangle^{-\f32}|\xi|^{\f32}|\widehat{\th}_{02}(t,\xi)|d\xi
\end{align*}
and we write the toy model for $X_2$ as follows:
\ben\label{eq:toyX_2}
\pa_tX_2-\nu\pa_{yy}X_2=C\ep \nu^{\al}\langle t\rangle^{-\f32}|\pa_y|^{\f32}\th_{02}, 
\een
where we also formally regard $\Delta_L$ as $\pa_{yy}$. 

For the term $(\wt{u}^2_{\neq}(\pa_y-t\pa_z)\th_{\neq})_0$, we also focus on the low-high interaction, namely, 
\begin{align*}
\mathcal{F}(\wt{u}^2_{\neq}(\pa_y-t\pa_z)\th_{\neq})_0
&=\cF\Big(\big((\wt{u}_{\neq}^2)_{\text{low}}((\pa_y-t\pa_z)\th_{\neq})_{\text{high}}\big)_0\Big)+\text{good terms}\\
&=\sum_{k\neq 0}\int_{|\eta|>2|kt|, \, |\xi-\eta|\leq \f19|\xi|} \widehat{\wt{u}}^2(t,-k,\eta-\xi)i(\xi-kt)\widehat{\th}(t,k,\xi)d\xi\\
&\quad+\text{good terms}.
\end{align*}
By \eqref{def f0 th0}, we formally regrad $X_1\sim X_2$, then 
\begin{align*}
|\mathcal{F}(\wt{u}^2_{\neq}(\pa_y-t\pa_z)\th_{\neq})_0|
\lesssim \sum_{k\neq 0}\int_{|\eta|>2|kt|, \, |\xi-\eta|\leq \f19|\xi|} \nu^{\al}\langle t\rangle^{-\f32}|\xi|^{\f12}|\widehat{X}_2(t,k,\xi)|d\xi,
\end{align*}
and we have the toy model for $\th_{0,2}$:
\ben\label{eq:toyth02}
\pa_t\th_{02}-\nu\pa_{yy}\th_{02}=C\ep\nu^{\al}\langle t\rangle^{-\f32}|\pa_y|^{\f12}X_2. 
\een

To close the energy estimate of the toy model \eqref{eq:toyX_2} and \eqref{eq:toyth02} in Sobolev space, we have to use the diffusion term and pay the smallness to control the derivative loss. Due to the diffusion effect, formally we regard $\pa_y\sim \nu^{-\f12}$. Then we have the following model:
\beno
\left\{\begin{aligned}
\pa_tX_2-\nu\pa_{yy}X_2=\ep \nu^{\al-\f34}\langle t\rangle^{-\f32}\th_{02},\\
\pa_t\th_{02}-\nu\pa_{yy}\th_{02}=\ep\nu^{\al-\f14}\langle t\rangle^{-\f32}X_2.
\end{aligned}\right.
\eeno
Thus, we have
\begin{align*}
\|X_2(t)\|_{H^s}^2+\int_0^t\|\pa_yX_2(\tau)\|_{H^s}^2d\tau
&\leq \|X_2(0)\|_{H^s}^2+\int_0^{t}\ep\nu^{\al-\f34}\langle \tau\rangle^{-\f32}d\tau\sup_{\tau\in [0,t]}\|X_2(\tau)\|_{H^s}\|\th_{02}(\tau)\|_{H^s}\\
&\leq \|X_2(0)\|_{H^s}^2+\ep\nu^{\al-\f34}\sup_{\tau\in [0,t]}\|X_2(\tau)\|_{H^s}\|\th_{02}(\tau)\|_{H^s},
\end{align*}
and
\begin{align*}
\|\th_{02}(t)\|_{H^s}^2+\int_0^t\|\pa_y\th_{02}(\tau)\|_{H^s}^2d\tau
&\leq \int_0^{t}\ep\nu^{\al-\f14}\langle \tau\rangle^{-\f32}d\tau\sup_{\tau\in [0,t]}\|X_2(\tau)\|_{H^s}\|\th_{02}(\tau)\|_{H^s}\\
&\leq \ep\nu^{\al-\f14}\sup_{\tau\in [0,t]}\|X_2(\tau)\|_{H^s}\|\th_{02}(\tau)\|_{H^s}.
\end{align*}
Thus it requires 
\beno
\ep\nu^{\al-\f34}\nu^{\al}\nu^{\b}\leq \nu^{2\al},\quad \ep\nu^{\al-\f14}\nu^{\al}\nu^{\b}\leq \nu^{2\b},
\eeno
which is $\al\geq \f12$ and $\b\geq \f34$. See the estimate of $J_{22}$ in Section 3.4 and the estimate of $K_{121}$ in Section 5.3 for the rigorous proof. 

\section{the interactions between the zero mode and the non-zero mode}\label{sec zero-nonzero mode}
In this section, we study the nonlinear interaction between the zero mode and the non-zero mode. We will prove \eqref{eq:I_1} and \eqref{eq:I_2} in Lemma \ref{lem: I} and \eqref{eq:J_1}, \eqref{eq:J_2} and \eqref{eq:J_3} in Lemma \ref{lem: J}. 
\subsection{Treatment of $I_1$} For $I_1$, by Young's inequality and the bootstrap hypotheses, we have
\begin{align*}
I_1
&=-\int\big[\cA\cN^{-1}(\wt{u}_0^1\pa_zf_{\neq})-\cA(\wt{u}_0^1\pa_zX_1)\big]\cA X_1dzdy\\
&\quad-\int\big[\cA(\wt{u}_0^1\pa_zX_1)-\wt{u}_0^1\pa_z\cA X_1\big]\cA X_1dzdy\\
&\eqdef I_1^{\text{com1}}+I_1^{\text{com2}},
\end{align*}
and here, we used that by integration parts,
\begin{align*}
\int(\wt{u}_0^1\pa_z\cA X_1)\cA X_1dzdy=0.
\end{align*}

For $I_1^{\text{com1}}$, by using $f_{\neq}=\cN X_1$ and  the fact that 
\begin{align*}
&\f{1}{(k^2+(\eta-kt)^2)^{\f14}}-\f{1}{(k^2+(\xi-kt)^2)^{\f14}}\\
&\lesssim(k^2+(\eta-kt)^2)^{-\f14}(k^2+(\xi-kt)^2)^{-\f14} \f{|\eta-\xi|}{(k^2+(\xi-kt)^2)^{\f14}+(k^2+(\eta-kt)^2)^{\f14}},
\end{align*}
we obtain 
\begin{align*}
| I_1^{\text{com1}}|&\lesssim\sum_{k\neq0}\int_{\xi,\eta}\cA_k^s(t,\eta)\Big|\f{1}{(k^2+(\eta-kt)^2)^{\f14}}-\f{1}{(k^2+(\xi-kt)^2)^{\f14}}\Big||\wh{\wt{u}}_0^1(t,\eta-\xi)|\\
 &\quad\times|k|(k^2+(\xi-kt)^2)^{\f14}|\wh{X}_1(t,k,\xi)|\cA_k^s(t,\eta)|X_1(t,k,\eta)|d\xi d\eta\\
 &\lesssim\sum_{k\neq0}\int_{\xi,\eta}\cA_k^s(t,\eta)(k^2+(\eta-kt)^2)^{-\f14}\f{|\eta-\xi||\wh{\wt{u}}_0^1(t,\eta-\xi)|}{(k^2+(\xi-kt)^2)^{\f14}+(k^2+(\eta-kt)^2)^{\f14}}\\
 &\quad\times|k||\wh{X}_1(t,k,\xi)|\cA_k^s(t,\eta)|X_1(t,k,\eta)|d\xi d\eta.
 \end{align*}
By \eqref{eq:A bound}, Young's inequality and the bootstrap hypotheses, we have 
\begin{align*}
  | I_1^{\text{com1}}|&\lesssim\sum_{k\neq0}\int_{\xi,\eta}\cA_k^s(t,\xi)|\eta-\xi||\wh{\wt{u}}_0^1(t,\eta-\xi)||\wh{X}_1(t,k,\xi)||k|(k^2+(\eta-kt)^2)^{-\f12}\\
 &\quad\times\cA_k^s(t,\eta)|X_1(t,k,\eta)|d\xi d\eta\\
 &\lesssim\|f_0\|_{H^s}\|\cA X_1\|_{L^2}\|\pa_z(-\D_L)^{-\f12}\cA X_1\|_{L^2}
 \leq C\ep\nu^{\f{1}{12}}(\mathfrak{ED})^{\f12}(\mathfrak{CK}_2)^{\f12}.
 \end{align*}

For $I_1^{\text{com2}}$, 
we get that by Lemma \ref{M lem},
\begin{align*}
|I_1^{\text{com2}}|&\lesssim\sum_{k\neq0}\int_{\xi,\eta}\big|\cA_k^s(t,\eta)-\cA_k^s(t,\xi)\big||\wh{\wt{u}}_0^1(t,\eta-\xi)||k||\wh{X}_1(t,k,\xi)|\cA_k^s(t,\eta)|\wh{X}_1(t,k,\eta)|d\xi d\eta\\
&\lesssim\sum_{k\neq0}\int_{\xi,\eta}\nu^{\f13}|\eta-\xi||k|^{\f23}|\wh{\wt{u}}_0^1(t,\eta-\xi)|\cA_k^s(t,\xi)|\wh{X}_1(t,k,\xi)|\cA_k^s(t,\eta)|\wh{X}_1(t,k,\eta)|d\xi d\eta\\
&\quad+\sum_{k\neq0}\int_{\xi,\eta}|\eta-\xi||\wh{\wt{u}}_0^1(t,\eta-\xi)|\cA_k^s(t,\xi)|\wh{X}_1(t,k,\xi)|\cA_k^s(t,\eta)|\wh{X}_1(t,k,\eta)|d\xi d\eta\\
&\quad +\sum_{k\neq0,|k|\leq|\eta-\xi|}\int_{\xi,\eta}\langle\eta-\xi\rangle^s|\wh{\wt{u}}_0^1(t,\eta-\xi)|e^{c\nu^{\f13}t}|k||\wh{X}_1(t,k,\xi)|\cA_k^s(t,\eta)|\wh{X}_1(t,k,\eta)|d\xi d\eta\\
&\quad+\sum_{k\neq0,|k|>|\eta-\xi|}\int_{\xi,\eta}|\eta-\xi||\wh{\wt{u}}_0^1(t,\eta-\xi)|\langle k,\xi\rangle^s|\wh{X}_1(t,k,\xi)|\cA_k^s(t,\eta)|\wh{X}_1(t,k,\eta)|d\xi d\eta\\
&\lesssim\nu^{\f13}\|\wt{u}_0^1\|_{H^3}\||D_z|^{\f13}\cA X_1\|_{L^2}^2+\|\wt{u}_0^1\|_{H^3}\|\cA X_1\|_{L^2}^2+\|\wt{u}_0^1\|_{H^s}\|e^{c\nu^{\f13}t}X_1\|_{H^3}\|\cA X_1\|_{L^2}\\
&\leq C\ep\nu^{\f12}\mathfrak{ED}+C\ep\nu^{\f16}\mathfrak{ED}\leq C\ep\nu^{\f16}\mathfrak{ED}.
\end{align*}

\subsection{Treatment of $I_2$}
By using the definition \eqref{def f0 th0}, we have
\begin{align}\label{u2}
|\wh{\wt{u}}_{\neq}^2(t,k,\eta)|\lesssim |k|^{\f12}(k^2+(\eta-kt)^2)^{-\f34}|\wh{X}_1(t,k,\eta)|,
\end{align}
and then, we get that for $I_2$, 
\begin{align*}
|I_2|
&\lesssim\sum_{k\neq0}\int_{\xi,\eta}\cA_k^s(t,\eta)|k|(k^2+(\eta-kt)^2)^{-\f14}(k^2+(\eta-\xi-kt)^2)^{-\f34}|\wh{X}_1(t,k,\eta-\xi)|\\
&\quad\times|(i\xi)\wh{f}_0(t,\xi)|\cA_k^s(t,\eta)|X_1(t,k,\eta)|d\xi d\eta\\
&\lesssim\sum_{k\neq0}\int_{\xi,\eta, |\xi|\leq|\eta-\xi|}\cA_k^s(t,\eta)|k|(k^2+(\eta-kt)^2)^{-\f14}(k^2+(\eta-\xi-kt)^2)^{-\f34}|\wh{X}_1(t,k,\eta-\xi)|\\
&\quad\times|(i\xi)\wh{f}_0(t,\xi)|\cA_k^s(t,\eta)|X_1(t,k,\eta)|d\xi d\eta\\
&\quad+\sum_{k\neq0}\int_{\xi,\eta, |\xi|>|\eta-\xi|}\cA_k^s(t,\eta)|k|(k^2+(\eta-kt)^2)^{-\f14}(k^2+(\eta-\xi-kt)^2)^{-\f34}|\wh{X}_1(t,k,\eta-\xi)|\\
&\quad\times|(i\xi)\wh{f}_0(t,\xi)|\cA_k^s(t,\eta)|X_1(t,k,\eta)|d\xi d\eta\\
&\eqdef I_{21}+I_{22}.
\end{align*}

For $I_{21}$, by using the fact that $\langle k,\eta\rangle^s\lesssim \langle k,\eta-\xi\rangle^s$ and for $k\neq0$,
\begin{align*}
\f{k^2+(\eta-kt)^2}{k^2+(\eta-\xi-kt)^2}\lesssim 1+\xi^2,
\end{align*}
we obtain that by Young's inequality and the bootstrap hypotheses, 
\begin{align}\label{I21}
|I_{21}|&\lesssim\sum_{k\neq0}\int_{|\xi|\leq|\eta-\xi|}\Big(\f{k^2+(\eta-kt)^2}{k^2+(\eta-\xi-kt)^2}\Big)^{\f14}\cA_k^s(t,\eta-\xi)|k|(k^2+(\eta-\xi-kt)^2)^{-\f12}\nonumber\\
&\quad\times|\wh{X}_1(t,k,\eta-\xi)||\xi\wh{f}_0(t,\xi)|(k^2+(\eta-kt)^2)^{-\f12}\cA_k^s(t,\eta)|X_1(t,k,\eta)|d\xi d\eta\nonumber\\
&\lesssim\sum_{k\neq0}\int_{|\xi|\leq|\eta-\xi|}\cA_k^s(t,\eta-\xi)|k|^{\f12}(k^2+(\eta-\xi-kt)^2)^{-\f12}|\wh{X}_1(t,k,\eta-\xi)|\nonumber\\
&\quad\times\langle\xi\rangle^3|\wh{f}_0(t,\xi)||k|^{\f12}(k^2+(\eta-kt)^2)^{-\f12}\cA_k^s(t,\eta)|X_1(t,k,\eta)|d\xi d\eta\nonumber\\
&\lesssim\|\pa_z(-\D_L)^{-\f12}\cA X_1\|_{L^2}^2\|f_0\|_{H^5}
\leq C\ep\nu^{\f14}\mathfrak{CK}_2.
\end{align}

For $I_{22}$, we have
\begin{align*}
|I_{22}|&\lesssim\sum_{k\neq0,|k|>|\xi|}\int_{|\eta-\xi|<|\xi|}\cA_k^s(t,\eta)|k|(k^2+(\eta-\xi-kt)^2)^{-\f34}|\wh{X}_1(t,k,\eta-\xi)|\\
&\quad\times|(i\xi)\wh{f}_0(t,\xi)|(k^2+(\eta-kt)^2)^{-\f14}\cA_k^s(t,\eta)|\wh{X}_1(t,k,\eta)|d\xi d\eta\\
&\quad+\sum_{k\neq0,|k|\leq|\xi|}\int_{|\eta-\xi|<|\xi|}\cA_k^s(t,\eta)|k|(k^2+(\eta-\xi-kt)^2)^{-\f34}|\wh{X}_1(t,k,\eta-\xi)|\\
&\quad\times|(i\xi)\wh{f}_0(t,\xi)|(k^2+(\eta-kt)^2)^{-\f14}\cA_k^s(t,\eta)|\wh{X}_1(t,k,\eta)|d\xi d\eta\\
&\eqdef I_{221}+I_{222}.
\end{align*}
For $I_{221}$, by $|k|>|\eta-\xi|$, we have $\langle k,\eta\rangle^s\lesssim\langle k,\eta-\xi\rangle^s+\langle k,\xi\rangle^s\lesssim\langle k\rangle^s$, and 
\begin{align*}
\Big(1+\big(t-\f{\eta-\xi}{k}\big)^2\Big)^{-\f34}\lesssim\langle t\rangle^{-\f32}.
\end{align*}
Then, we get that the bootstrap hypotheses, 
\begin{align}\label{I221}
|I_{221}|&\lesssim\sum_{k\neq0,|k|>|\xi|}\int_{|\eta-\xi|<|\xi|}e^{c\nu^{\f13}t}\langle k\rangle^s|k|(k^2+(\eta-\xi-kt)^2)^{-\f34}|\wh{X}_1(t,k,\eta-\xi)|\nonumber\\
&\quad\times|(i\xi)\wh{f}_0(t,\xi)|(k^2+(\eta-kt)^2)^{-\f14}\cA_k^s(t,\eta)|\wh{X}_1(t,k,\eta)|d\xi d\eta\nonumber\\
&\lesssim\sum_{k\neq0,|k|>|\xi|}\int_{|\eta-\xi|<|\xi|}e^{c\nu^{\f13}t}\langle k\rangle^s|k|^{-1}\Big(1+(t-\f{\eta-\xi}{k})^2\Big)^{-\f34}|\wh{X}_1(t,k,\eta-\xi)|\nonumber\\
&\quad\times|(i\xi)\wh{f}_0(t,\xi)|\cA_k^s(t,\eta)|\wh{X}_1(t,k,\eta)|d\xi d\eta\nonumber\\
&\lesssim\langle t\rangle^{-\f32}\|f_0\|_{H^3}\|\cA X_1\|_{L^2}^2\leq C\ep^3\nu^{\f54}\langle t\rangle^{-\f32}.
\end{align}
For $I_{222}$, by using $|\xi|^{\f12}\lesssim|\eta-\xi-kt|^{\f12}+|\eta-kt|^{\f12}$, $\langle k,\eta\rangle^s\lesssim\langle k,\xi\rangle^s\lesssim\langle \xi\rangle^s$, and \begin{align*}
\|(-\D_L)^{-\f34}X_1\|_{H^3}\lesssim\langle t\rangle^{-\f32}\|X_1\|_{H^{\f92}},
\end{align*}
we obtain
\begin{align*}
|I_{222}|&\lesssim\sum_{k\neq0,|k|\leq|\xi|}\int_{|\eta-\xi|<|\xi|}\cA_k^s(t,\eta)|k|(k^2+(\eta-\xi-kt)^2)^{-\f34}|\wh{X}_1(t,k,\eta-\xi)|\\
&\quad\times||\xi||\wh{f}_0(t,\xi)|(k^2+(\eta-kt)^2)^{-\f14}\cA_k^s(t,\eta)|\wh{X}_1(t,k,\eta)|d\xi d\eta\\
&\leq\sum_{k\neq0,|k|\leq|\xi|}\int_{|\eta-\xi|\leq|\xi|}|k|(k^2+(\eta-\xi-kt)^2)^{-\f12}e^{c\nu^{\f13}t}|\wh{X}_1(k,\eta-\xi)|\nonumber\\
&\quad\times\langle \xi\rangle^s|\xi|^{\f12}|\wh{f}_0(t,\xi)|(k^2+(\eta-kt)^2)^{-\f14}\cA_k^s(t,\eta)|\wh{X}_1(k,\eta)|d\xi d\eta\\
&\quad+\sum_{k\neq0,|k|\leq|\xi|}\int_{|\eta-\xi|\leq|\xi|}|k|(k^2+(\eta-\xi-kt)^2)^{-\f34}e^{c\nu^{\f13}t}|\wh{X}_1(t,k,\eta-\xi)|\\
&\quad\times\langle\xi\rangle^s|\xi|^{\f12}|\wh{f}_0(t,\xi)|\cA_k^s(t,\eta)|\wh{X}_1(t,k,\eta)|d\xi d\eta.
\end{align*}
Then, by Young's inequality, the interpolation inequality and the bootstrap hypotheses, we get
\begin{align}\label{I222}
|I_{222}|&\lesssim\|\pa_z(-\D_L)^{-\f12}e^{c\nu^{\f13}t}X_1\|_{H^2}\|\pa_y^{\f12}\langle D_y\rangle^sf_0\|_{L^2}\|(-\D_L)^{-\f14}\cA X_1\|_{L^2}\nonumber\\
&\quad+\|(-\D_L)^{-\f34}e^{c\nu^{\f13}t}X_1\|_{H^3}\|\pa_y^{\f12}\langle D_y\rangle^sf_0\|_{L^2}\|\cA X_1\|_{L^2}\nonumber\\
&\lesssim\|\pa_z(-\D_L)^{-\f12}\cA X_1\|_{L^2}\|\pa_y\langle D_y\rangle^sf_0\|_{L^2}^{\f12}\|\langle D_y\rangle^sf_0\|_{L^2}^{\f12}\|(-\D_L)^{-\f12}\cA X_1\|_{L^2}^{\f12}\|\cA X_1\|_{L^2}^{\f12}\nonumber\\
&\quad+\langle t\rangle^{-\f32}
\|\pa_y\langle D_y\rangle^sf_0\|_{L^2}^{\f12}\|\langle D_y\rangle^sf_0\|_{L^2}^{\f12}\|\cA X_1\|_{L^2}^2\nonumber\\
&\leq C\ep\nu^{\f38}(\mathfrak{CK}_2)^{\f34}\|\pa_y\langle D_y\rangle^sf_0\|_{L^2}^{\f12}+C\ep^{\f32}\nu^{\f98}\langle t\rangle^{-\f32}
\|\pa_y\langle D_y\rangle^sf_0\|_{L^2}^{\f12}.
\end{align}

Thus, combining \eqref{I21}, \eqref{I221} with \eqref{I222}, we obtain that 
\begin{align*}
|I_2|&\leq C\ep\nu^{\f14}\mathfrak{CK}_2+C\ep^3\nu^{\f54}\langle t\rangle^{-2}+C\ep\nu^{\f38}(\mathfrak{CK}_2)^{\f34}\|\pa_y\langle D_y\rangle^sf_0\|_{L^2}^{\f12}+C\ep^{\f32}\nu^{\f98}\langle t\rangle^{-\f32}
\|\pa_y\langle D_y\rangle^sf_0\|_{L^2}^{\f12}.
\end{align*}
\subsection{Treatment of $J_1$}
For $J_1$, we can rewrite it as:
\begin{align*}
J_1&=-2\g^2\s^{-1}\int\Big[\cA\cN(\wt{u}_0^1\pa_z^2\th_{\neq})-\cA(\wt{u}_0^1\cN\pa_z^2\th_{\neq})\Big]\cA X_2dzdy\\
&\quad-\s^{-1}\int\Big[\cA\dot{\cN}(\wt{u}_0^1\pa_zf_{\neq})-\cA(\wt{u}_0^1\dot{\cN}\pa_zf_{\neq})\Big]\cA X_2dzdy\\
&\quad+\int\Big[\cA(\wt{u}_0^1\pa_zX_2-\wt{u}_0^1\pa_z\cA X_2\Big]\cA X_2dzdy
+\int\wt{u}_0^1\pa_z\cA X_2\cA X_2dzdy\\
&=-2\g^2\s^{-1}\int\Big[\cA\cN(\wt{u}_0^1\pa_z^2\th_{\neq})-\cA(\wt{u}_0^1\cN\pa_z^2\th_{\neq})\Big]\cA X_2dzdy\\
&\quad-\s^{-1}\int\Big[\cA\dot{\cN}(\wt{u}_0^1\pa_zf_{\neq})-\cA(\wt{u}_0^1\dot{\cN}\pa_zf_{\neq})\Big]\cA X_2dzdy\\
&\quad+\int\Big[\cA(\wt{u}_0^1\pa_zX_2-\wt{u}_0^1\pa_z\cA X_2\Big]\cA X_2dzdy\\
&\eqdef J_1^{\text{com1}}+J_1^{\text{com2}}+J_1^{\text{com3}},
\end{align*}
whereas in the last second equality,   we get that by integration by parts,
\begin{align*}
\int\wt{u}_0^1\pa_z\cA X_2\cA X_2dzdy=0.
\end{align*}

By using the definition \eqref{def f0 th0}, we have
\begin{align}\label{th neq}
|\wh{\pa_z\th_{\neq}}(t,k,\eta)|&\lesssim |k|^{\f12}(k^2+(\eta-kt)^2)^{-\f14}|\wh{X}_2(t,k,\eta)|\nonumber\\
&\quad+|\eta-kt||k|^{\f12}(k^2+(\eta-kt)^2)^{-\f34}|\wh{X}_2(t,k,\eta)|\nonumber\\
&\lesssim |k|^{\f12}(k^2+(\eta-kt)^2)^{-\f14}\Big(|\wh{X}_2(t,k,\eta)|+|\wh{X}_2(t,k,\eta)|\Big).
\end{align}
Then, for $J_1^{\text{com1}}$, by using the fact that 
\begin{align*}
\big|(k^2+(\eta-kt)^2)^{\f14}-(k^2+(\xi-kt)^2)^{\f14}\big|\lesssim\f{|\eta-\xi|}{(k^2+(\eta-kt)^2)^{\f14}+(k^2+(\xi-kt)^2)^{\f14}},
\end{align*}
and $\langle k,\eta\rangle^s\lesssim\langle k,\eta-\xi\rangle^s+\langle k,\xi\rangle^s$, we get that
\begin{align*}
|J_1^{\text{com1}}|&\lesssim\sum_{k\neq0}\int_{\xi,\eta}\cA_k^s(t,\eta)|k|^{-\f12}\big|(k^2+(\eta-kt)^2)^{\f14}-(k^2+(\xi-kt)^2)^{\f14}\big||\wh{\wt{u}}_0^1(t,\eta-\xi)|\nonumber\\
&\quad\times|k|^{\f32}(k^2+(\xi-kt)^2)^{-\f14}\big(|\wh{X}_1|+|\wh{X}_2|\big)(t,k,\xi)\cA_k^s(t,\eta)|\wh{X}_2(t,k,\eta)|d\xi d\eta\nonumber\\
&\lesssim\sum_{k\neq0}\int_{\xi,\eta}\cA_k^s(t,\eta)|\eta-\xi||\wh{\wt{u}}_0^1(t,\eta-\xi)||k|(k^2+(\xi-kt)^2)^{-\f12}\big(|\wh{X}_1|+|\wh{X}_2|\big)(t,k,\xi)\nonumber\\
&\quad\times\cA_k^s(t,\eta)|\wh{X}_2(t,k,\eta)|d\xi d\eta\nonumber\\
&\lesssim\sum_{k\neq0}\int_{\xi,\eta}\cA_k^s(t,\eta-\xi)|\eta-\xi||\wh{\wt{u}}_0^1(t,\eta-\xi)||k|(k^2+(\xi-kt)^2)^{-\f12}\big(|\wh{X}_1|+|\wh{X}_2|\big)(t,k,\xi)\nonumber\\
&\quad\times\cA_k^s(t,\eta)|\wh{X}_2(t,k,\eta)|d\xi d\eta\nonumber\\
&+\sum_{k\neq0}\int_{\xi,\eta}|\eta-\xi||\wh{\wt{u}}_0^1(t,\eta-\xi)|\cA_k^s(t,\xi)|k|(k^2+(\xi-kt)^2)^{-\f12}\big(|\wh{X}_1|+|\wh{X}_2|\big)(t,k,\xi)\nonumber\\
&\quad\times\cA_k^s(t,\eta)|\wh{X}_2(t,k,\eta)|d\xi d\eta\nonumber\\
&\eqdef J_{11}^{\text{com1}}+J_{12}^{\text{com1}}.
\end{align*}
For $J_{11}^{\text{com1}}$, we get that by Young's inequality and the bootstrap hypotheses, 
\begin{align*}
|J_{11}^{\text{com1}}|&\lesssim\sum_{k\neq0,|k|\leq|\eta-\xi|}\int_{\xi,\eta}\langle \eta-\xi\rangle^s|\eta-\xi||\wh{\wt{u}}_0^1(t,\eta-\xi)|e^{c\nu^{\f13}t}|k|(k^2+(\xi-kt)^2)^{-\f12}\nonumber\\
&\quad\times\big(|\wh{X}_1|+|\wh{X}_2|\big)(t,k,\xi)\cA_k^s(t,\eta)|\wh{X}_2(t,k,\eta)|d\xi d\eta\nonumber\\
&\quad+\sum_{k\neq0,|k|>|\eta-\xi|}\int_{\xi,\eta}||\eta-\xi||\wh{\wt{u}}_0^1(t,\eta-\xi)|\cA_k^s(t,\xi)|k|(k^2+(\xi-kt)^2)^{-\f12}\nonumber\\
&\quad\times\big(|\wh{X}_1|+|\wh{X}_2|\big)(t,k,\xi)\cA_k^s(t,\eta)|\wh{X}_2(t,k,\eta)|d\xi d\eta\nonumber\\
&\lesssim\|f_0\|_{H^s}\|\cA X_2\|_{L^2}\Big(\|e^{c\nu^{\f13}t}\pa_z(-\D_L)^{-\f12}X_1\|_{H^2}+\|e^{c\nu^{\f13}t}\pa_z(-\D_L)^{-\f12}X_2\|_{H^2}\Big)\nonumber\\
&\quad+\|\wt{u}_0^1\|_{H^3}\|\cA X_2\|_{L^2}\Big(\|\pa_z(-\D_L)^{-\f12}\cA X_1\|_{L^2}+\|\pa_z(-\D_L)^{-\f12}\cA X_2\|_{L^2}\Big)\nonumber\\
&\lesssim\Big(\|f_0\|_{H^s}+\|\wt{u}_0^1\|_{H^s}\Big)\|\cA X_2\|_{L^2}\Big(\|\pa_z(-\D_L)^{-\f12}\cA X_1\|_{L^2}+\|\pa_z(-\D_L)^{-\f12}\cA X_2\|_{L^2}\Big)\nonumber\\
&\leq C\ep\nu^{\f{1}{12}}(\mathfrak{ED})^{\f12}(\mathfrak{CK}_2)^{\f12}+C\ep\nu^{\f13}(\mathfrak{ED})^{\f12}(\mathfrak{CK}_2)^{\f12}.
\end{align*}
And for $J_{12}^{\text{com1}}$, by Young's inequality and the bootstrap hypotheses, we have
\begin{align*}
|J_{12}^{\text{com1}}|&\lesssim\sum_{k\neq0}\int_{\xi,\eta}||\eta-\xi||\wh{\wt{u}}_0^1(t,\eta-\xi)|\cA_k^s(t,\xi)|k|(k^2+(\xi-kt)^2)^{-\f12}\nonumber\\
&\quad\times\big(|\wh{X}_1|+|\wh{X}_2|\big)(t,k,\xi)\cA_k^s(t,\eta)|\wh{X}_2(t,k,\eta)|d\xi d\eta\nonumber\\
&\lesssim\|\wt{u}_0^1\|_{H^3}\|\cA X_2\|_{L^2}\Big(\|\pa_z(-\D_L)^{-\f12}\cA X_1\|_{L^2}+\|\pa_z(-\D_L)^{-\f12}\cA X_2\|_{L^2}\Big)\nonumber\\
&\lesssim\|\wt{u}_0^1\|_{H^s}\|\cA X_2\|_{L^2}\Big(\|\pa_z(-\D_L)^{-\f12}\cA X_1\|_{L^2}+\|\pa_z(-\D_L)^{-\f12}\cA X_2\|_{L^2}\Big)\nonumber\\
&\leq C\ep\nu^{\f13}(\mathfrak{ED})^{\f12}(\mathfrak{CK}_2)^{\f12}.
\end{align*}
Thus, we obtain that 
\begin{align}\label{J11}
|J_1^{\text{com2}}|\leq C\ep\nu^{\f{1}{12}}(\mathfrak{ED})^{\f12}(\mathfrak{CK}_2)^{\f12}+C\ep\nu^{\f13}(\mathfrak{ED})^{\f12}(\mathfrak{CK}_2)^{\f12}\leq C\ep\nu^{\f{1}{12}}(\mathfrak{ED})^{\f12}(\mathfrak{CK}_2)^{\f12}.
\end{align}

For $J_1^{\text{com2}}$, by the definition \eqref{def f0 th0} and using the fact that
\begin{align*}
&(\eta-kt)(k^2+(\eta-kt)^2)^{-\f34}-(\xi-kt)(k^2+(\xi-kt)^2)^{-\f34}\\
&\lesssim|\eta-\xi|(k^2+(\xi-kt)^2)^{-\f34}+|\eta-\xi|(k^2+(\eta-kt)^2)^{-\f12}(k^2+(\xi-kt)^2)^{-\f14},
\end{align*}
 we get that by Young's inequality and the bootstrap hypotheses,
\begin{align}\label{J12}
|J_1^{\text{com2}}|&\lesssim\sum_{k\neq0}\int_{\xi,\eta}\cA_k^s(t,\eta)|k|^{\f12}\Big|(\eta-kt)(k^2+(\eta-kt)^2)^{-\f34}-(\xi-kt)(k^2+(\xi-kt)^2)^{-\f34}\Big|\nonumber\\
&\quad\times|\wh{\wt{u}}_0^1(t,\eta-\xi)||k|^{\f12}(k^2+(\xi-kt)^2)^{\f14}|\wh{X}_1(t,k,\xi)|\cA_k^s(t,\eta) |\wh{X}_2(t,k,\eta)|d\xi d\eta\nonumber\\
&\lesssim\sum_{k\neq0}\int_{\xi,\eta}\cA_k^s(t,\eta)|\eta-\xi||\wh{\wt{u}}_0^1(t,\eta-\xi)||k|(k^2+(\xi-kt)^2)^{-\f12}|\wh{X}_1(t,k,\xi)|\nonumber\\
&\quad\times\cA_k^s(t,\eta) |\wh{X}_2(t,k,\eta)|d\xi d\eta\nonumber\\
&\quad+\sum_{k\neq0}\int_{\xi,\eta}\cA_k^s(t,\eta)|\eta-\xi||\wh{\wt{u}}_0^1(t,\eta-\xi)||k||\wh{X}_1(t,k,\xi)|(k^2+(\eta-kt)^2)^{-\f12}\nonumber\\
&\quad\times\cA_k^s(t,\eta) |\wh{X}_2(t,k,\eta)|d\xi d\eta\nonumber\\
&\lesssim\|f_0\|_{H^s}\|\cA X_1\|_{L^2}\|\pa_z(-\D_L)^{-\f12}\cA X_2\|_{L^2}+\|f_0\|_{H^s}\|\cA X_2\|_{L^2}\|\pa_z(-\D_L)^{-\f12}\cA X_1\|_{L^2}\nonumber\\
&\leq  C\ep\nu^{\f{1}{12}}(\mathfrak{ED})^{\f12}(\mathfrak{CK}_2)^{\f12}.
\end{align}

For $J_1^{\text{com3}}$, by changing $X_1$ to be $X_2$ in $I_2^{\text{com2}}$, we get 
\begin{align}\label{J13}
|J_1^{\text{com3}}|\leq C\ep\nu^{\f16}\mathfrak{ED}.
\end{align}

From \eqref{J11}, \eqref{J12} and \eqref{J13}, we get 
\begin{align*}
|J_1|&\leq C\ep\nu^{\f{1}{12}}(\mathfrak{ED})^{\f12}(\mathfrak{CK}_2)^{\f12}+C\ep\nu^{\f16}(\mathfrak{ED})^{\f12}(\mathfrak{CK}_2)^{\f12}+C\ep\nu^{\f16}\mathfrak{ED}.
\end{align*}

\subsection{Treatment of $J_2$}
For $J_2$, we have
\begin{align*}
J_2
&=-2\g^2\s^{-1}\int\cA\pa_z\cN(\wt{u}_{\neq}^2\pa_y\th_{01})\cA X_2dzdy-2\g^2\s^{-1}\int\cA\pa_z\cN(\wt{u}_{\neq}^2\pa_y\th_{02})\cA X_2dzdy\\
&\eqdef J_{21}+J_{22}.
\end{align*}

For $J_{21}$, by \eqref {u2} and using the fact that 
\begin{align*}
(k^2+(\eta-kt)^2)^{\f14}\lesssim(k^2+(\eta-\xi-kt)^2)^{\f14}+|\xi|^{\f12},
\end{align*}
we get that by Young's inequality and the bootstrap hypotheses,
\begin{align*}
|J_{21}|&\lesssim\sum_{k\neq0}\int_{\xi,\eta}\cA_k^s(t,\eta)|k|(k^2+(\eta-kt)^2)^{\f14}(k^2+(\eta-\xi-kt)^2)^{-\f34}|\wh{X}_1(t,k,\eta-\xi)|\nonumber\\
&\quad\times |\xi\wh{\th}_{01}(t,\xi)|\cA_k^s(t,\eta)|\wh{X}_2(t,k,\eta)|d\xi d\eta\nonumber\\
&\lesssim\sum_{k\neq0}\int_{\xi,\eta}\cA_k^s(t,\eta)|k|(k^2+(\eta-\xi-kt)^2)^{-\f12}|\wh{X}_1(t,k,\eta-\xi)||\xi\wh{\th}_{01}(t,\xi)|\nonumber\\
&\quad\times\cA_k^s(t,\eta)|\wh{X}_2(t,k,\eta)|d\xi d\eta+\sum_{k\neq0}\int_{\xi,\eta}\cA_k^s(t,\eta)|k|(k^2+(\eta-\xi-kt)^2)^{-\f34}\nonumber\\
&\quad\times|\wh{X}_1(t,k,\eta-\xi)| |\xi|^{\f32}|\wh{\th}_{01}(t,\xi)|\cA_k^s(t,\eta)|\wh{X}_2(t,k,\eta)|d\xi d\eta\nonumber\\
&\lesssim\|\pa_z(-\D_L)^{-\f12}\cA X_1\|_{L^2}\|\th_{01}\|_{H^{s+1}}\|\cA X_2\|_{L^2}+\|\pa_z(-\D_L)^{-\f12}\cA X_1\|_{L^2}\|\th_{01}\|_{H^{s+2}}\|\cA X_2\|_{L^2}\\
&\leq C\ep\nu^{\f13}(\mathfrak{ED})^{\f12}(\mathfrak{CK}_2)^{\f12}.
\end{align*}

For $J_{22}$, by Young's inequality, the interpolation inequality and the bootstrap hypotheses, we have
\begin{align*}
|J_{22}|&\lesssim\sum_{k\neq0}\int_{\xi,\eta}\cA_k^s(t,\eta)|k|(k^2+(\eta-kt)^2)^{\f14}(k^2+(\eta-\xi-kt)^2)^{-\f34}|\wh{X}_1(t,k,\eta-\xi)|\nonumber\\
&\quad\times |\xi\wh{\th}_{02}(\xi)|\cA_k^s(t,\eta)|\wh{X}_2(t,k,\eta)|d\xi d\eta\nonumber\\
&\lesssim\|(-\D_L)^{\f14}\cA X_2\|_{L^2}\|(-\D_L)^{-\f14}\cA X_1\|_{L^2}\|\pa_y \th_{02}\|_{H^s}\nonumber\\
&\lesssim\|\na_L\cA X_2\|_{L^2}^{\f12}\|\cA X_2\|_{L^2}^{\f12}\|(-\D_L)^{-\f12}\cA X_1\|_{L^2}^{\f12}\|\cA X_1\|_{L^2}^{\f12}\|\pa_y\th_{02}\|_{H^s}\nonumber\\
&\leq C\ep\nu^{\f12}\mathfrak{D}^{\f14}(\mathfrak{CK}_2)^{\f14}\|\pa_y\th_{02}\|_{H^s}.
\end{align*}

Thus, we obtain
\begin{align*}
|J_2|\leq C\ep\nu^{\f13}(\mathfrak{ED})^{\f12}(\mathfrak{CK}_2)^{\f12}+C\ep\nu^{\f12}\mathfrak{D}^{\f14}(\mathfrak{CK}_2)^{\f14}\|\pa_y\th_{02}\|_{H^s}.
\end{align*}
\subsection{Treatment of $J_3$}
For $J_3$, similarly as the estimate of $I_2$, we get 
\begin{align*}
|J_3|&\leq C\ep\nu^{\f14}\mathfrak{CK}_2+C\ep^3\nu^{\f54}\langle t\rangle^{-2}+C\ep\nu^{\f38}(\mathfrak{CK}_2)^{\f32}\|\pa_y\langle D_y\rangle^sf_0\|_{L^2}^{\f12}+C\ep^{\f32}\nu^{\f98}\langle t\rangle^{-\f32}
\|\pa_y\langle D_y\rangle^sf_0\|_{L^2}^{\f12}.
\end{align*}

\section{the interactions between non-zero modes}\label{sec nonzero modes}
In this section, we study the nonlinear interactions between the non-zero modes. We mainly prove \eqref{eq:I_3I_4} in Lemma \ref{lem:  I} and \eqref{eq:J_4567} in Lemma \ref{lem:  J}.
\subsection{Treatment of $I_3$}
By using the fact that
\begin{align}\label{u1}
|\wh{\wt{u}}^1(t,k,\eta)|&\lesssim|\eta-kt||k|^{-\f12}(k^2+(\eta-kt)^2)^{-\f34}|\wh{X}_1(t,k,\eta)|\nonumber\\
&\lesssim|k|^{-\f12}(k^2+(\eta-kt)^2)^{-\f14}|\wh{X}_1(t,k,\eta)|,
\end{align}
and  $|l|^{\f12}\lesssim |k-l|^{\f12}+|k|^{\f12}$, we get that for $I_3$,
\begin{align*}
|I_3|
&\lesssim\sum_{k\neq0,l}\int_{\eta,\xi}\cA_k^s(t,\eta)|k-l|^{-\f12}\big((k-l)^2+(\eta-\xi-(k-l)t)^2\big)^{-\f14}|\wh{X}_1(t,k-l,\eta-\xi)|\nonumber\\
&\quad\times|l|^{\f12}(l^2+(\xi-lt)^2)^{\f14}|\wh{X}_1(t,l,\xi)|\cA_k^s(t,\eta)|k|^{\f12}(k^2+(\eta-kt)^2)^{-\f14}|\wh{X}_1(t,k,\eta)|d\xi d\eta\nonumber\\
&\lesssim\sum_{k\neq0,l}\int_{\eta,\xi}\cA_k^s(t,\eta)\big((k-l)^2+(\eta-\xi-(k-l)t)^2\big)^{-\f14}|\wh{X}_1(t,k-l,\eta-\xi)|\nonumber\\
&\quad\times(l^2+(\xi-lt)^2)^{\f14}|\wh{X}_1(t,l,\xi)|\cA_k^s(t,\eta)|k|^{\f12}(k^2+(\eta-kt)^2)^{-\f14}|\wh{X}_1(t,k,\eta)|d\xi d\eta\nonumber\\
&\quad+\sum_{k\neq0,l}\int_{\eta,\xi}\cA_k^s(t,\eta)|k-l|^{-\f12}\big((k-l)^2+(\eta-\xi-(k-l)t)^2\big)^{-\f14}|\wh{X}_1(t,k-l,\eta-\xi)|\nonumber\\
&\quad\times(l^2+(\xi-lt)^2)^{\f14}|\wh{X}_1(t,l,\xi)|\cA_k^s(t,\eta)|k|(k^2+(\eta-kt)^2)^{-\f14}|\wh{X}_1(t,k,\eta)|d\xi d\eta.
\end{align*}
Then by Young's inequality, the interpolation inequality  and the bootstrap hypotheses, we get that 
\begin{align*}
|I_3|&\lesssim\|(-\D_L)^{-\f14}\cA X_1\|_{L^2}\|(-\D_L)^{\f14}\cA X_1\|_{L^2}\|\pa_z^{\f12}(-\D_L)^{-\f14}\cA X_1\|_{L^2}\nonumber\\
&\quad+\|\pa_z^{-\f12}(-\D_L)^{-\f14}\cA X_1\|_{L^2}\|(-\D_L)^{\f14}\cA X_1\|_{L^2}\|\pa_z(-\D_L)^{-\f14}\cA X_1\|_{L^2}\nonumber\\
&\lesssim\|(-\D_L)^{-\f12}\cA X_1\|_{L^2}^{\f34}\|\na_L\cA X_1\|_{L^2}^{\f34}\|\cA X_1\|_{L^2}\|\pa_z(-\D_L)^{-\f12}\cA X_1\|_{L^2}^{\f12}\\
&\quad+\|\pa_z(-\D_L)^{-\f12}\cA X_1\|_{L^2}^{\f12}\|\cA X_1\|_{L^2}\|\na_L\cA X_1\|_{L^2}^{\f12}\|\pa_z(-\D_L)^{-\f12}\cA X_1\|_{L^2}^{\f12}\|\pa_z\cA X_1\|_{L^2}^{\f12}\\
&\lesssim\|\pa_z(-\D_L)^{-\f12}\cA X_1\|_{L^2}^{\f54}\|\na_L\cA X_1\|_{L^2}^{\f34}\|\cA X_1\|_{L^2}+\|\pa_z(-\D_L)^{-\f12}\cA X_1\|_{L^2}\|\cA X_1\|_{L^2}\|\na_L\cA X_1\|_{L^2}\\
&\leq C\ep\mathfrak{D}+C\ep\mathfrak{CK}_2.
\end{align*}
\subsection{Treatment of $I_4$}
For $I_4$, by the definition \eqref{def f0 th0}, \eqref{u2} and using the fact that
\begin{align*}
(l^2+(\xi-lt)^2)^{\f14}\lesssim(l^2+(\eta-\xi-lt)^2)^{\f14}+(l^2+(\eta-lt)^2)^{\f14},
\end{align*}
we have
\begin{align*}
|I_4|
&\lesssim\sum_{k\neq0}\int_{\eta,\xi}\cA_k^s(t,\eta)|k|^{\f12}(k^2+(\eta-kt)^2)^{-\f14}|k-l|^{\f12}((k-l)^2+(\eta-\xi-(k-l)t)^2)^{-\f34}\nonumber\\
&\quad\times|\wh{X}_1(t,k-l,\eta-\xi)||\xi-kt||l|^{-\f12}(l^2+(\xi-lt)^2)^{\f14}|\wh{X}_1(t,l,\xi)|\cA_k^s(t,\eta)|\wh{X}_1(t,k,\eta)|d\xi d\eta\nonumber\\
&\lesssim\sum_{k\neq0,l}\int_{\eta,\xi}\cA_k^s(t,\eta)|k-l|^{\f12}((k-l)^2+(\eta-\xi-(k-l)t)^2)^{-\f12}|\wh{X}_1(t,k-l,\eta-\xi)|\nonumber\\
&\quad\times|\xi-kt||l|^{-\f12}|\wh{X}_1(t,l,\xi)|\cA_k^s(t,\eta)|k|^{\f12}(k^2+(\eta-kt)^2)^{-\f14}|\wh{X}_1(t,k,\eta)|d\xi d\eta\nonumber\\
&\quad+\sum_{k\neq0,l}\int_{\eta,\xi}\cA_k^s(t,\eta)|k-l|^{\f12}((k-l)^2+(\eta-\xi-(k-l)t)^2)^{-\f34}|\wh{X}_1(t,k-l,\eta-\xi)|\nonumber\\
&\quad\times|\xi-kt||l|^{-\f12}|\wh{X}_1(t,l,\xi)|\cA_k^s(t,\eta)|k|^{\f12}|\wh{X}_1(t,k,\eta)|d\xi d\eta\\
&\eqdef I_{41}+I_{42}.
\end{align*}

For $I_{41}$, we obtain that by Young's inequality and the bootstrap hypotheses, 
\begin{align*}
|I_{41}|&\lesssim\|\pa_z(-\D_L)^{-\f12}\cA X_1\|_{L^2}\|\na_L\cA X_1\|_{L^2}\|\cA X_1\|_{L^2}\\
&\leq C\ep\mathfrak{D}^{\f12}(\mathfrak{CK}_2)^{\f12}.
\end{align*}

For $I_{42}$, by using $|k|^{\f12}\lesssim |k-l|^{\f12}+|l|^{\f12}$, by Young's inequality and the bootstrap hypotheses, we have
\begin{align*}
|I_{42}|&\lesssim\sum_{k\neq0,l}\int_{\eta,\xi}\cA_k^s(t,\eta)|k-l|((k-l)^2+(\eta-\xi-(k-l)t)^2)^{-\f34}|\wh{X}_1(t,k-l,\eta-\xi)|\nonumber\\
&\quad\times|\xi-kt||l|^{-\f12}|\wh{X}_1(t,l,\xi)|\cA_k^s(t,\eta)|\wh{X}_1(t,k,\eta)|d\xi d\eta\nonumber\\
&\quad+\sum_{k\neq0,l}\int_{\eta,\xi}\cA_k^s(t,\eta)|k-l|^{\f12}((k-l)^2+(\eta-\xi-(k-l)t)^2)^{-\f34}|\wh{X}_1(t,k-l,\eta-\xi)|\nonumber\\
&\quad\times|\xi-kt||\wh{X}_1(t,l,\xi)|\cA_k^s(t,\eta)|\wh{X}_1(t,k,\eta)|d\xi d\eta\nonumber\\
&\lesssim\|\pa_z(-\D_L)^{-\f12}\cA X_1\|_{L^2}\|\na_L\cA X_1\|_{L^2}\|\cA X_1\|_{L^2}\leq C\ep\mathfrak{D}^{\f12}(\mathfrak{CK}_2)^{\f12}.
\end{align*}

Thus, we deduce
\begin{align*}
|I_4|\leq C\ep\mathfrak{D}^{\f12}(\mathfrak{CK}_2)^{\f12}\leq C\ep\mathfrak{D}+C\ep\mathfrak{CK}_2.
\end{align*}

\subsection{Treatment of $J_4$}
For $J_4$,  we get that by the definition \eqref{def f0 th0},
\begin{align*}
J_4
&=-2\g^2\s^{-1}\int\cA\big(\pa_z\cN(\wt{u}_{\neq}^1\pa_z\th_{\neq})\big)\cA X_2dzdy\\
&=-\int\cA\big(\pa_z\cN(\wt{u}^1_{\neq}\cN^{-1}X_2)\big)\cA X_2dzdy-\s^{-1}\int\cA\big(\pa_z\cN(\wt{u}^1_{\neq}\dot{\cN}X_1)\big)\cA X_2dzdy\\
&=J_{41}+J_{42}.
\end{align*}

For $J_{41}$, by \eqref{u1} and using the fact that $|k|^{\f12}\lesssim|k-l|^{\f12}+|l|^{\f12}$, we get that by the interpolation  inequality and the bootstrap hypotheses, 
\begin{align}\label{J41}
|J_{41}|
&\lesssim\sum_{k\neq0,l}\int_{\xi,\eta}\cA_k^s(t,\eta)|k-l|^{-\f12}((k-l)^2+(\eta-\xi-(k-l)t)^2)^{-\f14}|\wh{X}_1(t,k-l,\eta-\xi)|\nonumber\\
&\quad\times|l|^{\f12}(l^2+(\xi-lt)^2)^{-\f14}|\wh{X}_2(t,l,\xi)||k|^{\f12}(k^2+(\eta-kt)^2)^{\f14}\cA_k^s(t,\eta)|\wh{X}_2(t,k,\eta)|d\xi d\eta\nonumber\\
&\lesssim\sum_{k\neq0,l}\int_{\xi,\eta}\cA_k^s(t,\eta)((k-l)^2+(\eta-\xi-(k-l)t)^2)^{-\f14}|\wh{X}_1(t,k-l,\eta-\xi)|\nonumber\\
&\quad\times|l|^{\f12}(l^2+(\xi-lt)^2)^{-\f14}|\wh{X}_2(t,l,\xi)|(k^2+(\eta-kt)^2)^{\f14}\cA_k^s(t,\eta)|\wh{X}_2(t,k,\eta)|d\xi d\eta\nonumber\\
&\quad+\sum_{k\neq0,l}\int_{\xi,\eta}\cA_k^s(t,\eta)|k-l|^{-\f12}((k-l)^2+(\eta-\xi-(k-l)t)^2)^{-\f14}|\wh{X}_1(t,k-l,\eta-\xi)|\nonumber\\
&\quad\times|l|(l^2+(\xi-lt)^2)^{-\f14}|\wh{X}_2(t,l,\xi)|(k^2+(\eta-kt)^2)^{\f14}\cA_k^s(t,\eta)|\wh{X}_2(t,k,\eta)|d\xi d\eta\nonumber\\
&\lesssim\|(-\D_L)^{-\f14}\cA X_1\|_{L^2}\|(-\D_L)^{\f14}\cA X_2\|_{L^2}\|\pa_z^{\f12}(-\D_L)^{-\f14}\cA X_2\|_{L^2}\nonumber\\
&\quad+\|\pa_z^{-\f12}(-\D_L)^{-\f14}\cA X_1\|_{L^2}\|(-\D_L)^{\f14}\cA X_2\|_{L^2}\|\pa_z(-\D_L)^{-\f14}\cA X_2\|_{L^2}\nonumber\\
&\lesssim\|(-\D_L)^{-\f12}\cA X_1\|_{L^2}^{\f34}\|\na_L\cA X_1\|_{L^2}^{\f14}\|\na_L\cA X_2\|_{L^2}^{\f12}\|\cA X_2\|_{L^2}\|\pa_z(-\D_L)^{-\f12}\cA X_2\|_{L^2}^{\f12}\nonumber\\
&\quad+\|\pa_z(-\D_L)^{-\f12}\cA X_1\|_{L^2}^{\f12}\|\cA X_1\|_{L^2}^{\f12}\|\cA X_2\|_{L^2}^{\f12}\|\na_L\cA X_2\|_{L^2}\|\pa_z(-\D_L)^{-\f12}\cA X_2\|_{L^2}^{\f12}\nonumber\\
&\leq C\ep\nu^{\f18}\mathfrak{D}^{\f38}(\mathfrak{CK}_2)^{\f58}+C\ep\mathfrak{D}^{\f12}(\mathfrak{CK}_2)^{\f12}.
\end{align}

Similarly as $J_{41}$,  we get that for $J_{42}$,
\begin{align}\label{J42}
|J_{42}|&\lesssim\sum_{k\neq0,l}\int_{\xi,\eta}\cA_k^s(t,\eta)|k-l|^{-\f12}\big((k-l)^2+(\eta-\xi-(k-l)t)^2\big)^{-\f14}|\wh{X}_1(k-l,\eta-\xi)|\nonumber\\
&\quad\times|l|^{\f12}\big(l^2+(\xi-lt)^2\big)^{-\f14}|\wh{X}_1(l,\xi)||k|^{\f12}(k^2+(\eta-kt)^2)^{\f14}\cA_k^s(t,\eta)|\wh{X}_2(k,\eta)|d\xi d\eta\\
&\lesssim\|(-\D_L)^{-\f14}\cA X_1\|_{L^2}\|(-\D_L)^{\f14}\cA X_2\|_{L^2}\|\pa_z^{\f12}(-\D_L)^{-\f14}\cA X_1\|_{L^2}\nonumber\\
&\quad+\|\pa_z^{-\f12}(-\D_L)^{-\f14}\cA X_1\|_{L^2}\|(-\D_L)^{\f14}\cA X_2\|_{L^2}\|\pa_z(-\D_L)^{-\f14}\cA X_1\|_{L^2}\nonumber\\
&\lesssim\|\pa_z(-\D_L)^{-\f12}\cA X_1\|_{L^2}^{\f54}\|\na_L\cA X_1\|_{L^2}^{\f14}\|\na_L\cA X_2\|_{L^2}^{\f12}\|\cA X_2\|_{L^2}^{\f12}\|\cA X_1\|_{L^2}^{\f12}\nonumber\\
&\quad+\|\pa_z(-\D_L)^{-\f12}\cA X_1\|_{L^2}\|\cA X_1\|_{L^2}^{\f12}\|\cA X_2\|_{L^2}^{\f12}\|\na_L\cA X_2\|_{L^2}^{\f12}\|\pa_z\cA X_1\|_{L^2}^{\f12}\nonumber\\
&\leq  C\ep\nu^{\f18}\mathfrak{D}^{\f38}(\mathfrak{CK}_2)^{\f58}+C\ep\mathfrak{D}^{\f12}(\mathfrak{CK}_2)^{\f12}.
\end{align}

Thus, from \eqref{J41}  and \eqref{J42}, we obtain
\begin{align*}
|J_4|\leq  C\ep\nu^{\f18}\mathfrak{D}^{\f38}(\mathfrak{CK}_2)^{\f58}+C\ep\mathfrak{D}^{\f12}(\mathfrak{CK}_2)^{\f12}\leq C\ep\mathfrak{D}+C\ep\mathfrak{CK}_2.
\end{align*}

\subsection{Treatment of $J_5$}
For $J_5$, by \eqref{u2}, \eqref{th neq} and using the fact that $|k|^{\f12}\lesssim |k-l|^{\f12}+|l|^{\f12}$, we get 
\begin{align*}
|J_5|&\lesssim\sum_{k\neq0,l}\int_{\xi,\eta}\cA_k^s(t,\eta)|k-l|^{\f12}\big((k-l)^2+(\eta-\xi-(k-l)t)^2\big)^{-\f34}|\wh{X}_1(t,k-l,\eta-\xi)|\\
&\quad\times|\xi-lt||l|^{-\f12}(l^2+(\xi-lt)^2)^{-\f14}\Big(|\wh{X}_1(t,l,\xi)|+|\wh{X}_2(t,l,\xi)|\Big)\\
&\quad\times|k|^{\f12}(k^2+(\eta-kt)^2)^{\f14}\cA_k^s(t,\eta)|\wh{X}_2(t,k,\eta)|d\xi d\eta\\
&\lesssim\sum_{k\neq0,l}\int_{\xi,\eta}\cA_k^s(t,\eta)|k-l|\big((k-l)^2+(\eta-\xi-(k-l)t)^2\big)^{-\f34}|\wh{X}_1(t,k-l,\eta-\xi)|\\
&\quad\times|\xi-lt||l|^{-\f12}(l^2+(\xi-lt)^2)^{-\f14}\Big(|\wh{X}_1(t,l,\xi)|+|\wh{X}_2(t,l,\xi)|\Big)\\
&\quad\times(k^2+(\eta-kt)^2)^{\f14}\cA_k^s(t,\eta)| \wh{X}_2(t,k,\eta)|d\xi d\eta\\
&\quad+\sum_{k\neq0,l}\int_{\xi,\eta}\cA_k^s(t,\eta)|k-l|^{\f12}\big((k-l)^2+(\eta-\xi-(k-l)t)^2\big)^{-\f34}|\wh{X}_1(t,k-l,\eta-\xi)|\\
&\quad\times|\xi-lt|(l^2+(\xi-lt)^2)^{-\f14}\Big(|\wh{X}_1(t,l,\xi)|+|\wh{X}_2(t,l,\xi)|\Big)\\
&\quad\times(k^2+(\eta-kt)^2)^{\f14}\cA_k^s(t,\eta)| \wh{X}_2(t,k,\eta)|d\xi d\eta.
\end{align*}
And then, by interpolation inequality and the bootstrap hypotheses, we obtain that 
\begin{align*}
|J_5|&\lesssim\|\pa_z(-\D_L)^{-\f12}\cA X_1\|_{L^2}\Big(\|(-\D_L)^{\f14}\cA X_1\|_{L^2}+\|(-\D_L)^{\f14}\cA X_2\|_{L^2}\Big)\|(-\D_L)^{\f14}\cA X_2\|_{L^2}\\
&\lesssim\|\pa_z(-\D_L)^{-\f12}\cA X_1\|_{L^2}\|\na_L\cA X_2\|_{L^2}\|\cA X_2\|_{L^2}\\
&\quad+\|\pa_z(-\D_L)^{-\f12}\cA X_1\|_{L^2}\|\na_L\cA X_1\|_{L^2}^{\f12}\|\cA X_1\|_{L^2}^{\f12}\|\na_L\cA X_2\|_{L^2}^{\f12}\|\cA X_2\|_{L^2}^{\f12}\\
&\leq C\ep\mathfrak{D}^{\f12}(\mathfrak{CK}_2)^{\f12}\leq C\ep\mathfrak{D}+C\ep\mathfrak{CK}_2.
\end{align*}
\subsection{Treatment of $J_6$}
For $J_6$ and $J_7$, we can get the same estimates as $I_3$ and $I_4$:\begin{align*}
|J_6|+|J_7|&\lesssim\|\pa_z(-\D_L)^{-\f12}\cA X_1\|_{L^2}^{\f34}\|\na_L\cA X_1\|_{L^2}^{\f14}\|\cA X_1\|_{L^2}^{\f12}\|\cA X_2\|_{L^2}^{\f12}\|\pa_z(-\D_L)^{-\f12}\cA X_2\|_{L^2}^{\f12}\\
&\quad+\|\pa_z(-\D_L)^{-\f12}\cA X_1\|_{L^2}^{\f12}\|\cA X_1\|_{L^2}\|\na_L\cA X_1\|_{L^2}^{\f12}\|\pa_z(-\D_L)^{-\f12}\cA X_2\|_{L^2}^{\f12}\|\na_L\cA X_2\|_{L^2}^{\f12}\\
&\quad+\|\pa_z(-\D_L)^{-\f12}\cA X_1\|_{L^2}\|\na_L\cA X_1\|_{L^2}\|\cA X_2\|_{L^2}\\
&\leq C\ep\mathfrak{D}+C\ep\mathfrak{CK}_2.
\end{align*}

\section{Energy estimates on zero modes}\label{section zero mode}
In this section, we mainly prove Proposition \ref{prop: zero}, namely, we study the energy estimates of $f_0$, $\wt{u}_0^1$ and $\th_{02}$. 

The most difficult part is the energy estimate on $\th_{02}$. There are two-type nonlinear interactions $\Gamma_1$ and $\Gamma_2$ (see \eqref{Gamma12}). The estimate of $\Gamma_1$ is easy to obtain as \eqref{cL1}. The low-high interaction $(\wt{u}_{\neq}^2)_{\text{low}}((\pa_y-t\pa_z)\th_{\neq})_{\text{high}}$ in $\Gamma_2$ is discussed in Section \ref{optimality}. The estimate of the high-low interaction $(\wt{u}_{\neq}^2)_{\text{high}}((\pa_y-t\pa_z)\th_{\neq})_{\text{low}}$ in $\Gamma_2$ is the most technical part. 
Indeed, for any fixed frequency $\eta$, when $t\in [t(\eta),2\eta]$, there is $k$ such that $k\eta>0$ and $t\in I_{k,\eta}$. At this critical time region $I_{k,\eta}$, the $k$th-mode of $\tilde{u}^2_{\neq}$, namely $\widehat{\wt{u}}^2_{\neq}(t,k,\xi)$ with $\xi\sim \eta$, causes a rapid growth of $\th_{02}$. Such growth of $\th_{02}$ varies in time and frequency. Thus we construct the time-dependent multiplier $m$, which allows us to control the growth of $\th_{02}$. 
\subsection{Energy estimates on  $f_0$}
For $f_0$, we have
\begin{lemma}\label{lem f0}
Under the bootstrap hypotheses, it holds that 
\begin{align*}
\f{d}{dt}\|f_0\|_{H^s}^2+2\nu\|\pa_yf_0\|_{H^s}^2\leq C\ep\nu^{\f14}\mathfrak{D}^{\f14}(\mathfrak{CK}_2)^{\f14}\|\pa_yf_0\|_{H^s}.
\end{align*}
\end{lemma}
\begin{proof}
Due to the fact that the zero mode $f_0$ satisfies
\begin{align*}
\pa_tf_0-\nu\pa_{yy}f_0=-\pa_y(\wt{u}_{\neq}^2f_{\neq})_0,
\end{align*}
we obtain that by using the energy method,
\begin{align*}
\f{d}{dt}\|\langle D_y\rangle^sf_0\|_{L^2}^2+2\nu\|\pa_y\langle D_y\rangle^sf_0\|_{L^2}^2&=-2\left\langle\langle D_y\rangle^s\pa_y(\wt{u}_{\neq}^2f_{\neq})_0, \langle D_y\rangle^sf_0\right\rangle_{L^2}.
\end{align*}
By \eqref{def f0 th0}, \eqref{u2}, using the interpolation inequality and the bootstrap hypotheses, we get
\begin{align*}
-2\left\langle\langle D_y\rangle^s\pa_y(\wt{u}_{\neq}^2f_{\neq})_0, \langle D_y\rangle^sf_0\right\rangle_{L^2}&\lesssim\|(-\D_L)^{-\f34}\cA  X_1\|_{L^2}\|(-\D_L)^{\f14}\cA  X_1\|_{L^2}\|\pa_yf_0\|_{H^s}\\
&\lesssim\|(-\D_L)^{-\f12}\cA  X_1\|_{L^2}\|\na_L\cA  X_1\|_{L^2}^{\f12}\|\cA  X_1\|_{L^2}^{\f12}\|\pa_yf_0\|_{H^s}\\
&\leq C\ep\nu^{\f14}\mathfrak{D}^{\f14}(\mathfrak{CK}_2)^{\f14}\|\pa_yf_0\|_{H^s}.
\end{align*}

Thus, we complete the proof of the lemma.
\end{proof}

\subsection{Energy estimates on  $\wt{u}_0^1$}
For $\wt{u}_0^1$, we mainly have the following lemma:
\begin{lemma}\label{lem u0}
Under the bootstrap hypotheses, there holds that 
\begin{align*}
\f{d}{dt}\|\wt{u}_0^1\|_{H^s}^2+2\nu\|\pa_y\wt{u}_0^1\|_{H^s}^2\leq C\ep\nu^{\f12}(\mathfrak{CK}_2)^{\f34}\|\pa_y\wt{u}_0^1\|_{H^s}^{\f12}+C\ep\nu^{\f{5}{12}}(\mathfrak{CK}_2)^{\f12}(\mathfrak{ED})^{\f14}\|\pa_y\wt{u}_0^1\|_{H^s}^{\f12}.
\end{align*}
\end{lemma}
\begin{proof}
From \eqref{eq: u0}, we obtain that by using the energy method,
\begin{align*}
\f{d}{dt}\|\langle D_y\rangle^s\wt{u}_0^1\|_{L^2}^2+2\nu\|\pa_y\langle D_y\rangle^s\wt{u}_0^1\|_{L^2}^2&=-2\left\langle\langle D_y\rangle^s\pa_y(\wt{u}_{\neq}^2\wt{u}^1_{\neq})_0, \langle D_y\rangle^s\wt{u}_0^1\right\rangle_{L^2}\eqdef\Xi.
\end{align*}
By using \eqref{u1}, \eqref{u2} and the fact that $|\eta|^{\f12}\lesssim |\xi-kt|^{\f12}+|\eta-\xi+kt|^{\f12}$, we have 
\begin{align*}
|\Xi|
&\lesssim \sum_{k\neq0}\int_{\xi,\eta}\langle \eta\rangle^s(k^2+(\xi-kt)^2)^{-\f34}|\wh{X}_1(t,k,\xi)|(k^2+(\eta-\xi+kt)^2)^{-\f14}\\
&\quad\times|\wh{X}_1(t,-k,\eta-\xi)|\langle \eta\rangle^s|\eta||\wh{\wt{u}}_0^1(t, \eta)|d\xi d\eta\\
&\lesssim \sum_{k\neq0}\int_{\xi,\eta}\langle \eta\rangle^s(k^2+(\xi-kt)^2)^{-\f12}|\wh{X}_1(t,k,\xi)|(k^2+(\eta-\xi+kt)^2)^{-\f14}\\
&\quad\times|\wh{X}_1(t,-k,\eta-\xi)|\langle \eta\rangle^s|\eta|^{\f12}|\wh{\wt{u}}_0^1(t, \eta)|d\xi d\eta\\
&\quad+\sum_{k\neq0}\int_{\xi,\eta}\langle \eta\rangle^s(k^2+(\xi-kt)^2)^{-\f34}|\wh{X}_1(t,k,\xi)||\wh{X}_1(t,-k,\eta-\xi)|\langle \eta\rangle^s|\eta|^{\f12}|\wh{\wt{u}}_0^1(t, \eta)|d\xi d\eta\\
&\eqdef\Xi_1+\Xi_2.
\end{align*}
Then, by the interpolation inequality and the bootstrap hypotheses, we get
\begin{align*}
|\Xi_1|&\lesssim\|(-\D_L)^{-\f12}\cA  X_1\|_{L^2}\|(-\D_L)^{-\f14}\cA  X_1\|_{L^2}\|\pa_y^{\f12}\langle D_y\rangle^s\wt{u}_0^1\|_{L^2}\\
&\lesssim\|(-\D_L)^{-\f12}\cA  X_1\|_{L^2}\|(-\D_L)^{-\f12}\cA  X_1\|_{L^2}^{\f12}\|\cA  X_1\|_{L^2}^{\f12}\|\pa_y\langle D_y\rangle^s\wt{u}_0^1\|_{L^2}^{\f12}\|\langle D_y\rangle^s\wt{u}_0^1\|_{L^2}^{\f12}\\
&\leq C\ep\nu^{\f12}(\mathfrak{CK}_2)^{\f34}\|\pa_y\langle D_y\rangle^s\wt{u}_0^1\|_{L^2}^{\f12},
\end{align*}
and 
\begin{align*}
|\Xi_2|&\lesssim\|(-\D_L)^{-\f34}\cA  X_1\|_{L^2}\|\cA  X_1\|_{L^2}\|\pa_y^{\f12}\langle D_y\rangle^s\wt{u}_0^1\|_{L^2}\\
&\lesssim \|(-\D_L)^{-\f34}\cA  X_1\|_{L^2}\|\cA  X_1\|_{L^2}\|\pa_y\langle D_y\rangle^s\wt{u}_0^1\|_{L^2}^{\f12}\|\langle D_y\rangle^s\wt{u}_0^1\|_{L^2}^{\f12}\\
&\leq C\ep\nu^{\f{5}{12}}(\mathfrak{CK}_2)^{\f12}(\mathfrak{ED})^{\f14}\|\pa_y\langle D_y\rangle^s\wt{u}_0^1\|_{L^2}^{\f12}.
\end{align*}

Thus, we complete the proof of the lemma.
\end{proof}

\subsection{Energy estimates on  $\th_{02}$}
To obtain the estimate \eqref{est: th02}, we mainly need to prove the following lemma.
\begin{lemma}\label{lem th02}
Under the bootstrap hypotheses, there holds that 
\begin{align*}
&\f{d}{dt}\Big\|\langle D_y\rangle^sm^{-1}\th_{02}\Big\|_{L^2}^2+2\nu\Big\|\pa_y\langle D_y\rangle^sm^{-1}\th_{02}\Big\|_{L^2}^2+2\Big\|\sqrt{\f{\pa_tm}{m}}\langle D_y\rangle^sm^{-1}\th_{02}\Big\|_{L^2}^2\\
&\quad\leq C\ep\nu^{\f{7}{12}}(\mathfrak{ED})^{\f12}(\mathfrak{CK}_2)^{\f12}+C\ep^{\f52}\nu^{\f54}\langle t\rangle^{-\f32}\mathfrak{D}^{\f14}+C\ep^3\nu^{\f74}\langle t\rangle^{-1}e^{-c\nu^{\f13}t}\\
&\quad\quad+C\ep\nu^{\f13}(\mathfrak{CK}_3)^{\f12}\Big\|\sqrt{\f{\pa_tm}{m}}m^{-1}\th_{02}\Big\|_{H^s}.
\end{align*}
\end{lemma}
\begin{proof}
From \eqref{eq: th02}, by applying $\langle D_y\rangle^s$ to $\eqref{eq: th02}_1$ and multiplying the equation by $\f{1}{m^2}\langle D_y\rangle^s\th_{02}$,  and then integrating over $\bbR$, we have that,
\begin{align}\label{Gamma12}
&\f{d}{dt}\Big\|\langle D_y\rangle^sm^{-1}\th_{02}\Big\|_{L^2}^2+2\nu\Big\|\pa_y\langle D_y\rangle^sm^{-1}\th_{02}\Big\|_{L^2}^2+2\Big\|\sqrt{\f{\pa_tm}{m}}\langle D_y\rangle^sm^{-1}\th_{02}\Big\|_{L^2}^2\nonumber\\
&\nonumber=-2\Big\langle\langle D_y\rangle^s\f{1}{m}(\wt{u}^1_{\neq}\pa_z\th_{\neq})_0,\langle D_y\rangle^sm^{-1}\th_{02}\Big\rangle_{L^2}\\
&\quad -2\Big\langle\langle D_y\rangle^s\f{1}{m}(\wt{u}^2_{\neq}(\pa_y-t\pa_z)\th_{\neq})_0,\langle D_y\rangle^sm^{-1}\th_{02}\Big\rangle_{L^2}\nonumber\\
&\eqdef\Gamma_1+\Gamma_2.
 \end{align}
 
On one hand, by \eqref{u1} and \eqref{th neq},  we obtain that by the interpolation inequality and the bootstrap hypotheses,
\begin{align}\label{cL1}
\Gamma_1&\lesssim\sum_k\int_{\xi,\eta}\f{\langle \eta\rangle^s}{m(t,\eta)}(k^2+(\xi-kt)^2)^{-\f14}\big|\wh{X}_1(t,k,\xi)\big|(k^2+(\eta-\xi+kt)^2)^{-\f14}\nonumber\\
&\quad\times\Big(\big|\wh{X}_1(t,-k,\eta-\xi)\big|+\big|\wh{X}_2(t,-k,\eta-\xi)\big|\Big)\langle \eta\rangle^s\f{|\wh{\th}_0(t,\eta)|}{m(t,\eta)}d\xi d\eta\nonumber\\
&\lesssim\|(-\D_L)^{-\f14}\cA X_1\|_{L^2}\|\langle D_y\rangle^sm^{-1}\th_{02}\|_{L^2}\Big(\|(-\D_L)^{-\f14}\cA X_2\|_{L^2}+\|(-\D_L)^{-\f14}\cA X_1\|_{L^2}\Big)\nonumber\\
&\lesssim\|(-\D_L)^{-\f12}\cA X_1\|_{L^2}^{\f12}\|\cA X_1\|_{L^2}^{\f12}\|(-\D_L)^{-\f12}\cA X_2\|_{L^2}^{\f12}\|\cA X_2\|_{L^2}^{\f12}\|\langle D_y\rangle^sm^{-1}\th_{02}\|_{L^2}\nonumber\\
&\quad+\|(-\D_L)^{-\f12}\cA X_1\|_{L^2}\|\cA X_1\|_{L^2}\|\langle D_y\rangle^sm^{-1}\th_{02}\|_{L^2}\nonumber\\
&\leq C\ep\nu^{\f{7}{12}}(\mathfrak{ED})^{\f12}(\mathfrak{CK}_2)^{\f12}.
\end{align}

On the other hand, by \eqref{u2} and \eqref{th neq}, we have
\begin{align*}
\Gamma_2&\lesssim\int_{\xi,\eta}\sum_{k}\f{\langle \eta\rangle^s}{m(t,\eta)}(k^2+(\xi-kt)^2)^{-\f34}|\wh{X}_1(t,k,\xi)|(k^2+(\eta-\xi+kt)^2)^{\f14}\\
&\quad\times|\wh{X}_2(t,-k,\eta-\xi)|\langle \eta\rangle^s\f{|\wh{\th}_{02}(t,\eta)|}{m(t,\eta)}d\xi d\eta\\
&\quad+\int_{\xi,\eta}\sum_{k}\f{\langle \eta\rangle^s}{m(t,\eta)}(k^2+(\xi-kt)^2)^{-\f34}|\wh{X}_1(t,k,\xi)|(k^2+(\eta-\xi+kt)^2)^{\f14}\\
&\quad\times|\wh{X}_1(t,-k,\eta-\xi)|\langle \eta\rangle^s\f{|\wh{\th}_{02}(t,\eta)|}{m(t,\eta)}d\xi d\eta\\
&\eqdef K_1+K_2.
\end{align*}
In the following, we only need to estimate $K_1$, and $K_2$ can be similarly estimated. 

Due to 
we have
\begin{align*}
|K_1|&\lesssim\int_{\xi,\eta}\sum_k\mathds{1}_{|k|>|\eta|}\f{\langle \eta\rangle^s}{m(t,\eta)}(k^2+(\xi-kt)^2)^{-\f34}|\wh{X}_1(t,k,\xi)|\\
&\quad\times(k^2+(\eta-\xi+kt)^2)^{\f14}|\wh{X}_2(t,-k,\eta-\xi)|\langle \eta\rangle^s\f{|\wh{\th}_{02}(t,\eta)|}{m(t,\eta)}d\xi d\eta\\
&\quad+\int_{\xi,\eta}\sum_k\mathds{1}_{|k|\leq|\eta|}\f{\langle \eta\rangle^s}{m(t,\eta)}(k^2+(\xi-kt)^2)^{-\f34}|\wh{X}_1(t,k,\xi)|\\
&\quad\times(k^2+(\eta-\xi+kt)^2)^{\f14}|\wh{X}_2(t,-k,\eta-\xi)|\langle \eta\rangle^s\f{|\wh{\th}_{02}(t,\eta)|}{m(t,\eta)}d\xi d\eta\\
&\eqdef K_{11}+K_{12}.
\end{align*}
For $K_{11}$, by using the fact that 
\begin{align}\label{000}
(k^2+(\eta-\xi+kt)^2)^{\f14}\lesssim \langle t\rangle^{\f12}\langle k,\eta-\xi\rangle^{\f12},
\end{align}
and for $|k|>|\eta|$,
\begin{align*}
(k^2+(\xi-kt)^2)^{-\f34}&\lesssim(k^2+(\eta-kt)^2)^{-\f34}\f{\langle k,\eta-\xi\rangle^{\f32}}{|k|^{\f32}}\lesssim \langle t\rangle^{-\f32}\f{\langle k,\eta-\xi\rangle^{\f32}}{|k|^3},
\end{align*}
and 
by Young's inequality and the bootstrap hypotheses, we have that for $s\geq 6$,
\begin{align}\label{K11}
|K_{11}|&\lesssim\int_{\xi,\eta}\sum_k\mathds{1}_{|k|>|\eta|}\f{\langle k\rangle^s}{m(t,\eta)}|\wh{X}_1(t,k,\xi)|\f{\langle k,\eta-\xi\rangle^2}{|k|^3}\langle t\rangle^{-1}\nonumber\\
&\quad\times|\wh{X}_2(t,-k,\eta-\xi)|\langle \eta\rangle^s\f{|\wh{\th}_{02}(t,\eta)|}{m(t,\eta)}d\xi d\eta\nonumber\\
&\lesssim\langle t\rangle^{-1}\|X_1\|_{H^{\f{s}{2}+2}}\|X_2\|_{H^{\f{s}{2}+3}}\|\langle D_y\rangle^sm^{-1}\th_{02}\|_{L^2}\nonumber\\
&\lesssim\langle t\rangle^{-1}e^{-c\nu^{\f13}t}\|\cA X_1\|_{L^2}\|\cA X_2\|_{L^2}\|m^{-1}\th_{02}\|_{H^s}\nonumber\\
&\leq C\ep^3\nu^{\f74}\langle t\rangle^{-1}e^{-c\nu^{\f13}t}.
\end{align}

For $K_{12}$, we need to estimate it more carefully. By the high and low frequency decomposition, we have
\begin{align*}
|K_{12}|&\lesssim\int_{\xi,\eta}\mathds{1}_{|\xi|\leq|\eta-\xi|}\sum_k\mathds{1}_{|k|\leq|\eta|}\f{\langle \eta\rangle^s}{m(t,\eta)}(k^2+(\xi-kt)^2)^{-\f34}|\wh{X}_1(t,k,\xi)|\\
&\quad\times(k^2+(\eta-\xi+kt)^2)^{\f14}|\wh{X}_2(t,-k,\eta-\xi)|\langle \eta\rangle^s\f{|\wh{\th}_{02}(t,\eta)|}{m(t,\eta)}d\xi d\eta\\
&\quad+\int_{\xi,\eta}\mathds{1}_{|\xi|>|\eta-\xi|}\sum_k\mathds{1}_{|k|\leq|\eta|}\f{\langle \eta\rangle^s}{m(t,\eta)}(k^2+(\xi-kt)^2)^{-\f34}|\wh{X}_1(t,k,\xi)|\\
&\quad\times(k^2+(\eta-\xi+kt)^2)^{\f14}|\wh{X}_2(t,-k,\eta-\xi)|\langle \eta\rangle^s\f{|\wh{\th}_{02}(t,\eta)|}{m(t,\eta)}d\xi d\eta\\
&\eqdef K_{121}+K_{122}.
\end{align*}

For $K_{121}$, by using the fact that $\langle \eta\rangle^s\leq \langle k,\eta-\xi\rangle^s+\langle k,\xi\rangle^s\leq\langle k, \eta-\xi\rangle^s$, Young's inequality and 
\begin{align*}
\|(-\D_L)^{-\f34}X_1\|_{H^2}\lesssim\langle t\rangle^{-\f32}\|X_1\|_{H^{\f72}},
\end{align*}
we obtain that  by the interpolation inequality and the bootstrap hypotheses, 
\begin{align}\label{K121}
|K_{121}|
&\lesssim\|(-\D_L)^{-\f34}X_1\|_{H^2}\|(-\D_L)^{\f14}\cA X_2\|_{L^2}\|m^{-1}\th_{02}\|_{H^s}\nonumber\\
&\lesssim\langle t\rangle^{-\f32}\|X_1\|_{H^{\f72}}\|\na_L\cA X_2\|_{L^2}^{\f12}\|\cA X_2\|_{L^2}^{\f12}\|m^{-1}\th_{02}\|_{H^s}\nonumber\\
&\lesssim\langle t\rangle^{-\f32}\|\cA X_1\|_{L^2}^{\f32}\|\na_L\cA X_2\|_{L^2}^{\f12}\|m^{-1}\th_{02}\|_{H^s}\nonumber\\
&\leq C\ep^{\f52}\nu^{\f54}\langle t\rangle^{-\f32}\mathfrak{D}^{\f14}.
\end{align}

For $K_{122}$, by using the fact that 
\begin{align*}
(k^2+(\xi-kt)^2)^{-\f34}&\lesssim(k^2+(\eta-kt)^2)^{-\f34}\f{\langle k,\eta-\xi\rangle^{\f32}}{|k|^{\f32}},
\end{align*}
we get that
\begin{align*}
|K_{122}|&\lesssim\int_{\xi,\eta}\sum_{k}\mathds{1}_{t\geq 2|\eta|}\mathds{1}_{|\xi|>|\eta-\xi|}\mathds{1}_{|k|\leq|\eta|}\f{\langle \eta\rangle^s}{m(t,\eta)}(k^2+(\xi-kt)^2)^{-\f34}|\wh{X}_1(t,k,\xi)|\\
&\quad\times(k^2+(\eta-\xi+kt)^2)^{\f14}|\wh{X}_2(t,-k,\eta-\xi)|\langle \eta\rangle^s\f{|\wh{\th}_{02}(t,\eta)|}{m(t,\eta)}d\xi d\eta\\
&\quad+\int_{\xi,\eta}\sum_{k}\mathds{1}_{t(\eta)\leq t<2|\eta|}\mathds{1}_{|\xi|>|\eta-\xi|}\mathds{1}_{|k|\leq|\eta|}\f{\langle \eta\rangle^s}{m(t,\eta)}(k^2+(\xi-kt)^2)^{-\f34}|\wh{X}_1(t,k,\xi)|\\
&\quad\times(k^2+(\eta-\xi+kt)^2)^{\f14}|\wh{X}_2(t,-k,\eta-\xi)|\langle \eta\rangle^s\f{|\wh{\th}_{02}(t,\eta)|}{m(t,\eta)}d\xi d\eta\\
&\quad+\int_{\xi,\eta}\sum_{k}\mathds{1}_{t<t(\eta)}\mathds{1}_{|\xi|>|\eta-\xi|}\mathds{1}_{|k|\leq|\eta|}\f{\langle \eta\rangle^s}{m(t,\eta)}(k^2+(\xi-kt)^2)^{-\f34}|\wh{X}_1(t,k,\xi)|\\
&\quad\times(k^2+(\eta-\xi+kt)^2)^{\f14}|\wh{X}_2(t,-k,\eta-\xi)|\langle \eta\rangle^s\f{|\wh{\th}_{02}(t,\eta)|}{m(t,\eta)}d\xi d\eta\\
&\eqdef\cK_1+\cK_2+\cK_3.
\end{align*}

 For $\cK_1$, by using $\langle \eta\rangle^s\leq\langle k,\eta\rangle^s\leq\langle k,\eta-\xi\rangle^s+\langle k,\xi\rangle^s\leq\langle k, \xi\rangle^s$, \eqref{000} and
\begin{align*}
(k^2+(\xi-kt)^2)^{-\f34}&\lesssim(k^2+(\eta-kt)^2)^{-\f34}\f{(k^2+(\eta-kt)^2)^{\f34}}{(k^2+(\xi-kt)^2)^{\f34}}\\
&\lesssim(k^2+(\eta-kt)^2)^{-\f34}\f{\langle k,\eta-\xi\rangle^{\f32}}{|k|^{\f32}},
\end{align*}
we get that by the interpolation inequality and the bootstrap hypotheses,
\begin{align}\label{cK1}
|\cK_1|&\lesssim\int_{\xi,\eta}\sum_{k}\mathds{1}_{t\geq 2|\eta|}\mathds{1}_{|\xi|>|\eta-\xi|}\mathds{1}_{|k|\leq|\eta|}\f{\langle k,\xi\rangle^s}{m(t,\eta)}(k^2+(\eta-kt)^2)^{-\f34}\f{\langle k,\eta-\xi\rangle^{\f32}}{|k|^{\f32}}|\wh{X}_1(t,k,\xi)|\nonumber\\
&\quad\times(k^2+(\eta-\xi+kt)^2)^{\f14}|\wh{X}_2(t,-k,\eta-\xi)|\langle \eta\rangle^s\f{|\wh{\th}_{02}(t,\eta)|}{m(t,\eta)}d\xi d\eta\nonumber\\
&\lesssim\int_{\xi,\eta}\sum_{k}\mathds{1}_{t\geq 2|\eta|}\mathds{1}_{|\xi|>|\eta-\xi|}\mathds{1}_{|k|\leq|\eta|}\langle t\rangle^{-\f32}\langle k,\xi\rangle^s|\wh{X}_1(t,k,\xi)|\langle t\rangle^{\f12}\f{\langle k,\eta-\xi\rangle^{2}}{|k|^{\f32}}\nonumber\\
&\quad\times|\wh{X}_2(t,-k,\eta-\xi)|\langle \eta\rangle^s\f{|\wh{\th}_{02}(t,\eta)|}{m(t,\eta)}d\xi d\eta\nonumber\\
&\lesssim\langle t\rangle^{-1}\|\cA X_1\|_{H^s}\|X_2\|_{H^4}\|\langle D_y\rangle^sm^{-1}\th_{02}\|_{L^2}\nonumber\\
&\lesssim\langle t\rangle^{-1}e^{-c\nu^{\f13}t}\|\cA X_1\|_{L^2}\|\cA X_2\|_{L^2}\|\langle D_y\rangle^sm^{-1}\th_{02}\|_{L^2}\nonumber\\
&\leq C\ep^3\nu^{\f74}\langle t\rangle^{-1}e^{-c\nu^{\f13}t}.
\end{align}

For $\cK_3$, we have
\begin{align*}
|\cK_3|&\lesssim\int_{\xi,\eta}\sum_{k}\mathds{1}_{t< t(\eta)}\mathds{1}_{|\xi|>|\eta-\xi|}\mathds{1}_{1\leq|k|\leq\f{1}{10}\sqrt{|\eta|}}\f{\langle \eta\rangle^s}{m(t,\eta)}(k^2+(\xi-kt)^2)^{-\f34}\\
&\quad\times|\wh{X}_1(t,k,\xi)|(k^2+(\eta-\xi+kt)^2)^{\f14}|\wh{X}_2(t,-k,\eta-\xi)|\langle \eta\rangle^s\f{|\wh{\th}_{02}(t,\eta)|}{m(t,\eta)}d\xi d\eta\\
&\quad+\int_{\xi,\eta}\sum_{k}\mathds{1}_{t< t(\eta)}\mathds{1}_{|\xi|>|\eta-\xi|}\mathds{1}_{\f{1}{10}\sqrt{|\eta|}<|k|\leq|\eta|}\f{\langle \eta\rangle^s}{m(t,\eta)}(k^2+(\xi-kt)^2)^{-\f34}\\
&\quad\times|\wh{X}_1(t,k,\xi)|(k^2+(\eta-\xi+kt)^2)^{\f14}|\wh{X}_2(t,-k,\eta-\xi)|\langle \eta\rangle^s\f{|\wh{\th}_{02}(t,\eta)|}{m(t,\eta)}d\xi d\eta\\
&\eqdef \cK_{31}+\cK_{32}.
\end{align*}

For $1\leq|k|\leq\f{1}{10}\sqrt{|\eta|}$, and then by $|\eta-kt|\geq\f12|\eta|\gtrsim\langle t\rangle^2$, we have 
\begin{align*}
(k^2+(\xi-kt)^2)^{-\f34}&\lesssim(k^2+(\eta-kt)^2)^{-\f34}\f{\langle k,\eta-\xi\rangle^{\f32}}{|k|^{\f32}}\\
&\lesssim\f{\langle k,\eta-\xi\rangle^{\f32}}{|\eta-kt|^{\f32}}\lesssim\f{\langle k,\eta-\xi\rangle^{\f32}}{\langle t\rangle^{3}}.
\end{align*}
And then, by using $\langle \eta\rangle^s\lesssim \langle k,\xi\rangle^s$, \eqref{000} and  the bootstrap hypotheses, we obtain 
\begin{align*}
|\cK_{31}|
&\lesssim
\int_{\xi,\eta}\sum_{k}\mathds{1}_{t< t(\eta)}\mathds{1}_{|\xi|>|\eta-\xi|}\mathds{1}_{1\leq|k|\leq\f{1}{10}\sqrt{|\eta|}}\f{\langle k,\xi\rangle^s}{m(t,\eta)}\wh{X}_1(t,k,\xi)\\
&\quad\times \langle t\rangle^{-\f52}\langle k,\eta-\xi\rangle^2|\wh{X}_2(t,-k,\eta-\xi)|\langle \eta\rangle^s\f{|\wh{\th}_{02}(t,\eta)|}{m(t,\eta)}d\xi d\eta\\
&\lesssim \langle t\rangle^{-\f52}e^{-c\nu^{\f13}t}\|\cA X_1\|_{L^2}\|X_2\|_{H^4}\|\langle D_y\rangle^sm^{-1}\th_{02}\|_{L^2}\\
&\lesssim \langle t\rangle^{-\f52}e^{-c\nu^{\f13}t}\|\cA X_1\|_{L^2}\|\cA X_2\|_{L^2}\|\langle D_y\rangle^sm^{-1}\th_{02}\|_{L^2}\\
&\leq C\ep^3\nu^{\f74}\langle t\rangle^{-\f52}e^{-c\nu^{\f13}t}.
\end{align*} 

For $\f{1}{10}\sqrt{|\eta|}<|k|\leq|\eta|$, by using the fact that $\langle \eta\rangle^s\lesssim \langle k,\xi\rangle^s$, \eqref{000} and 
\begin{align*}
(k^2+(\xi-kt)^2)^{-\f34}\lesssim k^{-\f32}\lesssim \langle t\rangle^{-\f32},
\end{align*}
we obtain that  by Young's inequality and the bootstrap hypotheses, 
\begin{align*}
|\cK_{31}|
&\lesssim\int_{\xi,\eta}\sum_{k}\mathds{1}_{t< t(\eta)}\mathds{1}_{|\xi|>|\eta-\xi|}\mathds{1}_{\f{1}{10}\sqrt{|\eta|}<|k|<|\eta|}\f{\langle k,\xi\rangle^s}{m(t,\eta)}|\wh{X}_1(t,k,\xi)| \langle t\rangle^{-1}\\
&\quad\times\langle k,\eta-\xi\rangle^{\f12}|\wh{X}_2(t,-k,\eta-\xi)|\langle \eta\rangle^s\f{|\wh{\th}_{02}(t,\eta)|}{m(t,\eta)}d\xi d\eta\\
&\lesssim \langle t\rangle^{-1}e^{-c\nu^{\f13}t}\|\cA X_1\|_{L^2}\|X_2\|_{H^{\f52}}\|\langle D_y\rangle^sm^{-1}\th_{02}\|_{L^2}\\
&\lesssim \langle t\rangle^{-1}e^{-c\nu^{\f13}t}\|\cA X_1\|_{L^2}\|\cA X_2\|_{L^2}\|\langle D_y\rangle^sm^{-1}\th_{02}\|_{L^2}\\
&\leq C\ep^3\nu^{\f74}\langle t\rangle^{-1}e^{-c\nu^{\f13}t}.
\end{align*} 

Thus, we get that 
\begin{align}\label{cK3}
|\cK_3|\leq C\ep^3\nu^{\f74}\langle t\rangle^{-\f52}e^{-c\nu^{\f13}t}+C\ep^3\nu^{\f74}\langle t\rangle^{-1}e^{-c\nu^{\f13}t}\leq C\ep^3\nu^{\f74}\langle t\rangle^{-1}e^{-c\nu^{\f13}t}.
\end{align}

Finally, we estimate the term $\cK_2$. For $t(\eta)\leq t\leq 2|\eta|$, we have $t\in[\f{2|\eta|}{2j+1},\f{2|\eta|}{2j-1}]$ for some $j\in[1,E(\sqrt{|\eta|})]$.  And then, for $\cK_2$, we have
\begin{align*}
|\cK_2|&\lesssim\int_{\xi,\eta}\sum_{k\neq j}\mathds{1}_{t(\eta)\leq t\leq 2|\eta|}\mathds{1}_{t\in I_{j,\eta}}\mathds{1}_{|\xi|>|\eta-\xi|}\mathds{1}_{|k|\leq|\eta|}\f{\langle k,\xi\rangle^s}{m(t,\eta)}(k^2+(\eta-kt)^2)^{-\f34}\\
&\quad\times\langle k,\eta-\xi\rangle^{\f32}|\wh{X}_1(t,k,\xi)|(k^2+(\eta-\xi+kt)^2)^{\f14}|\wh{X}_2(t,-k,\eta-\xi)|\langle \eta\rangle^s\f{|\wh{\th}_{02}(t,\eta)|}{m(t,\eta)}d\xi d\eta\\
&\quad+\int_{\xi,\eta}\sum_{j=1}^{E(\sqrt{|\eta|})}\mathds{1}_{t\in I_{j,\eta}}\mathds{1}_{|\xi|>|\eta-\xi|}\mathds{1}_{|k|\leq|\eta|}\f{\langle k,\xi\rangle^s}{m(t,\eta)}(k^2+(\eta-kt)^2)^{-\f34}\\
&\quad\times\langle k,\eta-\xi\rangle^{\f32}|\wh{X}_1(t,k,\xi)|(k^2+(\eta-\xi+kt)^2)^{\f14}|\wh{X}_2(t,-k,\eta-\xi)|\langle \eta\rangle^s\f{|\wh{\th}_{02}(t,\eta)|}{m(t,\eta)}d\xi d\eta\\
&\eqdef \cK_{21}+\cK_{22}.
\end{align*}

When $j\neq k$, we have $|t-\f{\eta}{k}|\gtrsim \f{|\eta|}{j^2}$, and for $1\leq |j|\leq 4|k|$,
\begin{align*}
(k^2+(\eta-kt)^2)^{-\f34}\lesssim k^{-\f32}(1+\f{|\eta|^2}{j^4})^{-\f34}
\lesssim k^{-\f32}(1+\f{t^2}{k^2})^{-\f34}\lesssim\langle t\rangle^{-\f32},
\end{align*}
and for $|j|>4|k|$, by using $|\f{\eta}{k}-t|\geq \f{2|\eta|}{2j-1}\geq t$, we have
\begin{align*}
(k^2+(\eta-kt)^2)^{-\f34}&=k^{-\f32}(1+(t-\f{\eta}{k})^2)^{-\f34}\\&\lesssim k^{-\f32}(1+t^2)^{-\f34}\lesssim k^{-\f32}\langle t\rangle^{-\f32}.
\end{align*}
Thus, for $\cK_{21}$, we get that
\begin{align*}
|\cK_{21}|&\lesssim\int_{\xi,\eta}\sum_{k\neq j}\mathds{1}_{t(\eta)\leq t\leq 2|\eta|}\mathds{1}_{t\in I_{j,\eta}}\mathds{1}_{|\xi|>|\eta-\xi|}\mathds{1}_{|k|\leq|\eta|}\langle k,\xi\rangle^s|\wh{X}_1(t,k,\xi)|\langle t\rangle^{-1}\langle k,\eta-\xi\rangle^{2}\\
&\quad\times|\wh{X}_2(t,-k,\eta-\xi)|\langle \eta\rangle^s\f{|\wh{\th}_{02}(t,\eta)|}{m(t,\eta)}d\xi d\eta\\
&\lesssim\langle t\rangle^{-1}e^{-c\nu^{\f13}t}\|\cA X_1\|_{L^2}\|X_2\|_{H^4}\|\langle D_y\rangle^sm^{-1}\th_{02}\|_{L^2}\\
&\lesssim\langle t\rangle^{-1}e^{-c\nu^{\f13}t}\|\cA X_1\|_{L^2}\|\cA X_2\|_{L^2}\|\langle D_y\rangle^sm^{-1}\th_{02}\|_{L^2}\\
&\leq C\ep^3\nu^{\f74}\langle t\rangle^{-1}e^{-c\nu^{\f13}t}.
\end{align*}

When $j=k$, we have
\begin{align*}
|\cK_{22}|&\lesssim\int_{\xi,\eta}\sum_{j=1}^{E(\sqrt{|\eta|})}\mathds{1}_{t\in I_{j,\eta}}\mathds{1}_{|\xi|>|\eta-\xi|}\mathds{1}_{|j|\leq|\eta|}\f{\langle j,\xi\rangle^s}{m(t,\eta)}(j^2+(\xi-jt)^2)^{-\f38}|\wh{X}_1(t,j,\xi)|\f{\langle j,\eta-\xi\rangle^{\f34}}{|j|^{\f32}}\\
&\quad\times(j^2+(\eta-\xi+jt)^2)^{\f14}|\wh{X}_2(t,-j,\eta-\xi)|\big(1+(t-\f{\eta}{j})^2\big)^{-\f38}\langle \eta\rangle^s\f{|\wh{\th}_{02}(t,\eta)|}{m(t,\eta)}d\xi d\eta,
\end{align*}
and then, we get that by using the definition of the multiplier $m$ and the bootstrap hypotheses,
\begin{align*}
|\cK_{22}|&\lesssim\int_{\xi,\eta}\sum_{j=1}^{E(\sqrt{|\eta|})}\mathds{1}_{t\in I_{j,\eta}}\mathds{1}_{|\xi|>|\eta-\xi|}\mathds{1}_{|j|\leq|\eta|}\f{\langle j,\xi\rangle^s}{m(t,\eta)}(j^2+(\xi-jt)^2)^{-\f38}|\wh{X}_1(t,j,\xi)|\f{\langle j,\eta-\xi\rangle^{\f34}}{|j|^{\f12}}\\
&\quad\times(j^2+(\eta-\xi+jt)^2)^{\f14}|\wh{X}_2(t,-j,\eta-\xi)||j|^{-1}\big(1+(t-\f{\eta}{j})^2\big)^{-\f38}\langle \eta\rangle^s\f{|\wh{\th}_{02}(t,\eta)|}{m(t,\eta)}d\xi d\eta\\
&\lesssim\int_{\xi,\eta}\sum_{j=1}^{E(\sqrt{|\eta|})}\mathds{1}_{t\in I_{j,\eta}}\mathds{1}_{|\xi|>|\eta-\xi|}\mathds{1}_{|j|\leq|\eta|}\f{\langle j,\xi\rangle^s}{m(t,\eta)}(j^2+(\xi-jt)^2)^{-\f38}|\wh{X}_1(t,j,\xi)|\\
&\quad\times t^{\f12}\langle j,\eta-\xi\rangle^{\f54}|\wh{X}_2(t,-j,\eta-\xi)|\sqrt{\f{\pa_tm_j}{m_j}}(t,j,\eta)\langle \eta\rangle^s\f{|\wh{\th}_{02}(t,\eta)|}{m(t,\eta)}d\xi d\eta\\
&\lesssim t^{\f12}e^{-c\nu^{\f13}t}\Big\|\sqrt{\f{\pa_tm}{m}}m^{-1}\th_{02}\Big\|_{H^s}\|(-\D_L)^{-\f38}\cA X_2\|_{L^2}\|X_1\|_{H^4}\\
&\lesssim t^{\f12}e^{-c\nu^{\f13}t}\Big\|\sqrt{\f{\pa_tm}{m}}m^{-1}\th_{02}\Big\|_{H^s}\|(-\D_L)^{-\f38}\cA X_2\|_{L^2}\|\cA X_1\|_{L^2}\\
&\leq C\ep\nu^{\f13}(\mathfrak{CK}_3)^{\f12}\Big\|\sqrt{\f{\pa_tm}{m}}m^{-1}\th_{02}\Big\|_{H^s}.
\end{align*}
Thus, we get
\begin{align}\label{cK2}
|\cK_2|\leq C\ep^3\nu^{\f74}\langle t\rangle^{-1}e^{-c\nu^{\f13}t}+C\ep\nu^{\f13}(\mathfrak{CK}_3)^{\f12}\Big\|\sqrt{\f{\pa_tm}{m}}m^{-1}\th_{02}\Big\|_{H^s}
\end{align}

From  \eqref{cL1}, \eqref{K11} - 
\eqref{cK2}, 
we complete the proof of Lemma \ref{lem th02}.
\end{proof}

Finally, we conclude that Proposition \ref{prop: zero} follows directly from Lemma \ref{lem f0}, Lemma \ref{lem u0} and Lemma \ref{lem th02} with $\ep$ small enough.

\section*{Acknowledgements}
C. Zhai's work is supported by a grant from the China Scholarship Council and this work was done when C. Zhai was visiting the center SITE, NYU Abu Dhabi. She appreciates the hospitality of NYU.

\end{CJK*}
\end{document}